%% file: main.tex
\newif\ifArXiV

\ArXiVtrue 

\ifArXiV
\documentclass{article}
\else
\documentclass[authoryear,preprint]{elsarticle}
\fi

\usepackage{enumitem}   
\usepackage{algorithm}
\usepackage{algpseudocode}
\usepackage[utf8]{inputenc}
\usepackage[T1]{fontenc}
\usepackage{amsmath,amssymb,mathtools}
\usepackage{amsthm}
\usepackage{amsfonts}
\usepackage{graphicx}
\usepackage{array}
\usepackage{hyperref}
\usepackage{float}
\usepackage{booktabs}
\usepackage[a4paper, left=1in,right=1in,top=1in,bottom=1in]{geometry}
\usepackage{xcolor}
\usepackage{xspace}
\usepackage{todonotes}

\usepackage{subcaption}
\usepackage{longtable}
\usepackage{pdflscape}
\usepackage[onehalfspacing]{setspace}
\usepackage[normalem]{ulem}



\newcommand{\mytag}[1]{(\hypertarget{#1}{\mathrm{#1}})}
\newcommand{\myrefmath}[1]{(\hyperlink{#1}{\mathrm{#1}})}
\newcommand{\myref}[1]{\textnormal{(\hyperlink{#1}{#1}})}

\newcommand{\epsConstraint}{$\varepsilon$\text{-constraint}\xspace}
\newcommand{\dimension}{\binom{n}{2}}
\newcommand{\fourindex}{h}

\newcommand{\BB}{B\&B\xspace}

\newcommand{\SRFLP}{SRFLP\xspace}
\newcommand{\MOSRFLP}{MOSRFLP\xspace}
\newcommand{\BOSRFLP}{BOSRFLP\xspace}

\DeclareMathOperator*{\argmin}{arg min}
\DeclareMathOperator*{\argmax}{arg max}
\DeclareMathOperator{\rank}{rank}
\DeclareMathOperator{\diag}{diag}

\newtheorem{theorem}{Theorem}

\newtheorem{definition}[theorem]{Definition}

\newtheorem{observation}[theorem]{Observation}
\newtheorem{proposition}[theorem]{Proposition}

\ifArXiV
\usepackage[affil-it]{authblk}
\usepackage[authoryear, round]{natbib}
\newenvironment{frontmatter}{}{}
\newenvironment{keyword}{\small \textbf{Keywords:}}{}
\let\address\affil
\fi

\makeatletter
\newenvironment{subalgorithms}{%
  \refstepcounter{algorithm}%
  \protected@edef\theparentalgorithm{\thealgorithm}%
  \setcounter{parentalgorithm}{\value{algorithm}}%
  \setcounter{algorithm}{0}%
  \def\thealgorithm{\theparentalgorithm\alph{algorithm}}%
  \ignorespaces
}{%
  \setcounter{algorithm}{\value{parentalgorithm}}%
  \ignorespacesafterend
}
\newcounter{parentalgorithm}
\makeatother

\begin{document}
\begin{frontmatter}
\title{An semidefinite programming-based $\varepsilon$-constraint method for the bi-objective single-row facility layout problem} 
\ifArXiV
\author[1]{Christof Brandstetter}
\author[2]{Elisabeth Gaar}
\author[1]{Markus Sinnl}
\affil[1]{Institute of Business Analytics and Technology Transformation/JKU Business School,
Johannes Kepler University Linz, Austria}
\affil[2]{Institute of Mathematics, University of Augsburg, Germany}
\date{}
\maketitle

\fi

\begin{abstract}
In this work, we introduce a multi-objective version of the well-known single-row facility layout problem (\SRFLP). In the \SRFLP, a set of one-dimensional facilities should be placed along a single line such that the weighted sum of the center-to-center distances of each pair of facilities is minimized. In our multi-objective extension, there are multiple such weighted-sum objectives which we consider under the concept of Pareto optimality.

We develop a solution algorithm based on the $\varepsilon$-constraint method to solve the bi-objective \SRFLP. Many existing works on the $\varepsilon$-constraint method use integer linear programming (ILP) solvers in a black-box fashion for solving the problems at the individual iterations of the method. In contrast to that, we use our own branch-and-bound procedure based on semidefinite programming (SDP), as SDP relaxations are known to be more effective for solving the \SRFLP than linear programming relaxations of ILPs. This allows us to propose several enhancements procedures for our $\varepsilon$-constraint approach, such as non-binary branching and reusing of nodes within the branch-and-bound trees, which are usually not possible when using black-box solvers. We present a computational study to demonstrate the effectiveness of our solution approach and its enhancements. 

\end{abstract}
\end{frontmatter}

\section{Introduction}\label{sec:Introduction}
The \emph{single-row facility layout problem (\SRFLP)}, introduced by \citet{SIMMONS1969}, is a well-known optimization problem with numerous practical applications, such as arrangement of rooms in hospitals, shelves in supermarkets, and office layouts as mentioned in \citet{SIMMONS1969}. Further applications are the assignment of airplanes to gates in airport terminals as noted in \citet{SURYANARAYANAN1991}, the organization of machines in flexible manufacturing systems as mentioned in~\citet{Heragu1988} and storage optimization as described in~\citet{Picard1981}.

In this work, we introduce the \emph{multi-objective single-row facility layout problem (\MOSRFLP)}, which extends the \SRFLP to multiple objectives, allowing decision-makers to take into account various different real-life aspects. In particular, the \SRFLP optimizes only one objective function, while there might be multiple conflicting objectives in applications. For instance, in multi-product manufacturing systems, often each product requires a distinct production process, leading to as many objective functions as products when optimizing the production process for all products. Similarly, warehouse designs must compromise between optimal inbound, outbound and intermediate storage layouts, while airport gate assignments that minimize customer transfer times might result in longer walking times for the personnel. To model these applications, we consider the \MOSRFLP: Given some one-dimensional facilities of certain lengths and $p$ weights between each pair of facilities, the~\MOSRFLP seeks an arrangement of the facilities on a line, such that the sums of the $p$ weighted center-to-center distances of each pair of facilities are minimized.

For $p=1$ the \MOSRFLP coincides with the \SRFLP.
For $p>1$, i.e., in the actual multi-objective setting, there is usually not one single solution that minimizes all $p$ sums simultaneously. We therefore consider the concept of Pareto optimality, see e.g., \citet{Ehrgott2005}. A solution is Pareto optimal if no objective function value can be improved without deteriorating the value of another objective function. The corresponding vector of objective function values of such a Pareto optimal solution is called a nondominated point. Note that multiple Pareto optimal solutions can result in the same nondominated point.
The goal of the \MOSRFLP is then to determine the nondominated set, i.e., the set of all nondominated points. 

\subsection{Contribution and outline}\label{sec:contribution}
In this work, we introduce the \MOSRFLP and develop an $\varepsilon$-constraint approach to solve the problem for $p=2$, i.e., to determine the nondominated set of the bi-objective \SRFLP. Our approach uses a handcrafted semidefinite programming (SDP)-based branch-and-bound (\BB) algorithm to solve the single-objective subproblems at each iteration of the $\varepsilon$-constraint method. To the best of our knowledge, this is the first work to solve the subproblems within the $\varepsilon$-constraint method with an SDP-based \BB algorithm. Usually, these subproblems are formulated as integer linear programs (ILPs) and solved in a black-box fashion using solvers like CPLEX or Gurobi. Our SDP-based \BB algorithm allows us to investigate how the $\varepsilon$-constraint method and the \BB algorithm can be more effectively integrated, rather than being treated independently. In particular, we introduce a scheme for reusing nodes of the \BB trees. Other enhancements include various branching strategies, sequential variable fixing, node relaxation reduction and solution pooling.

We tested various settings of our algorithm on 145 literature-based and randomly generated instances with up to 20 facilities. With our best performing setting, using all enhancements, we obtained the nondominated sets for 96\% of the instances, whereas our basic setting obtained the nondominated sets for only 48\% of the instances. Furthermore, we compare the performance of our algorithm with ILP-based approaches, which obtained the nondominated sets for at most 69\% of the instances. Our best performing setting was up to 180 times faster in terms of runtime than the basic setting and up to 90 times faster than the best ILP-based approach.

The paper is structured as follows: In the remainder of this section, we discuss previous and related work. In Section~\ref{sec:background}, we introduce the basics of multi-objective integer programming, formally define the~\MOSRFLP, and present several mathematical formulations for both the single-objective and the multi-objective~\SRFLP.
In Section~\ref{sec:algorithm}, we first present our $\varepsilon$-constraint method and then our SDP-based \BB algorithm together with various branching strategies, including non-binary branching, and a method to reduce the dimension of the SDPs that need to be solved. In Section~\ref{sec:MO-Enhance}, we introduce multi-objective enhancements for our solution algorithm, i.e., a technique to reuse the \BB search tree from a previous iteration and a solution pooling approach to obtain better feasible solutions. In Section~\ref{sec:comp-results}, our computational study is detailed. Finally, Section~\ref{sec:conclusion} concludes the paper.

\subsection{Literature review}\label{sec:literature}
The \SRFLP was introduced by \citet{SIMMONS1969} as the one-dimensional space allocation problem and is a generalization of the $\mathcal{NP}$-hard minimum linear arrangement problem, hence it is $\mathcal{NP}$-hard as well (see, e.g., \citet{Garey1976} for further details). In \citet{SIMMONS1969}, a combinatorial \BB algorithm for solving the \SRFLP was presented, which allowed the author to obtain an optimal solution for instances with up to eleven facilities. A first ILP-based solution approach was proposed by \citet{LOVE1976}. However, the linear programming (LP) relaxation of this formulation is weak, and the authors managed to solve only one instance with five facilities to optimality. In \citet{Picard1981}, a dynamic programming approach was developed and instances with up to 14 facilities could be solved.

In the 2000s, a new direction for tackling the \SRFLP emerged, namely solution algorithms based on SDP relaxations.
In \citet{anjos_semidefinite_2005} and \citet{anjos_computing_2008}, SDP-based approaches were considered in order to obtain lower bounds on the optimal objective function value, and a heuristic that extracts feasible solutions from the optimal solutions of SDP relaxations was presented. In particular,~\citet{anjos_computing_2008} improved the initial relaxation used in \citet{anjos_semidefinite_2005} by adding so-called triangle inequalities. With this approach, instances with up to 30 facilities could be solved.
\citet{hungerlander_semidefinite_2013} introduced several different sets of valid inequalities and solved the respective SDP relaxations with a tailored bundle method, allowing them to solve strengthened SDP relaxations containing a large set of valid inequalities within reasonable time. They were able to solve instances with up to 42 facilities to optimality. 

Later \citet{Schwiddessen2020, Schwiddessen2022} proposed to use so-called $k$-clique inequalities for various values of~$k$, where for $k=3$ one obtains the triangle inequalities. The SDP relaxations were solved using a penalty approach, and instances with up to 81 facilities could be solved to optimality. We note that aside from just considering SDP relaxations, in the Master thesis \citet{Schwiddessen2020} also a \BB based on the considered SDP relaxations is described, however, no computational results are provided.

Finally, aside from these SDP-based approaches, at the same time, new ILP-based approaches emerged, with the one of \citet{amaral_new_2009}, which uses so-called betweenness variables, being the computationally most successful ILP-based approach. Instances with up to 35 facilities could be solved to optimality with this approach. 

In addition to exact methods, there exists exhaustive work on heuristics for the \SRFLP. An early metaheuristic used for the \SRFLP was the simulated annealing approach by \citet{Heragu1992}. To the best of our knowledge, the current state-of-the-art heuristic approaches for the \SRFLP are a memetic algorithm with simulated annealing based local search procedure by \citet{Tang2023}, an oscillation based simulated annealing approach by \citet{CHEN2025} and a window approach metaheuristic by \citet{PAMMER2026}.

There also exists some previous work on multi-objective extensions to the \SRFLP. However, these works did not consider the \MOSRFLP as we define it in this work. For example, \citet{Lenin2018} (and earlier references therein) considered a multi-objective multi-product variant of an \SRFLP-like problem. The problem consists of determining a
common linear machine sequence for multiple products that require different operation sequences and machine types, with a
limited number of duplicate machine types.
The goal is to minimize the total cost of material flow, the total number of machines used and the total investment cost. The authors used a genetic algorithm to obtain heuristic solutions for this problem and compared these solutions using an average fitness score.
The performance of the authors' algorithm was presented in a computational study on instances from the literature and on randomly generated instances with at most 15 facility types.

Another multi-objective variant of the \SRFLP was considered in \citet{Azadeh2011}. The authors simulated an injection molding process using a fuzzy simulation-fuzzy data envelopment analysis algorithm, where an injection head filled up different molds, which are aligned in a single row. The goal was to identify the layout that optimized multiple performance measures like average waiting time, average machine utilization, and average time in the system.

In \citet{Ouhoud2025}, the authors developed a mixed integer linear programming (MILP) formulation for a variant of the \SRFLP that allows facilities to be rotated. They considered two different objective functions, namely minimizing the total flow distance and maximizing a total closeness ratio. In a small computational study, the authors solved the MILP formulation considering each objective function separately for an instance with five facilities. The authors then also solved the MILP formulation using a normalized weighted-sum objective function across three different sets of weights to analyze the trade-off between the objective functions.

\section{Theoretical background and problem formulations}\label{sec:background} 
In this section, we first provide a brief overview of multi-objective integer programming. Then we give a formal definition of the \MOSRFLP, followed by several formulations for the \SRFLP. Based on these formulations, we conclude this section with a formulation for the \MOSRFLP.

\subsection{Multi-objective integer programs}\label{sec:MOIP}
A general \emph{multi-objective integer program (MOIP)} can be defined as 
\begin{equation*}    
    \begin{aligned}
        \min \quad & f(x)\\
        \text{s.t.} \quad  &x \in \mathcal{X} \cap \mathbb Z^n,
    \end{aligned}
\end{equation*}
where $\mathcal{X}$ $\subseteq \mathbb{R}^n$ for some $n \in \mathbb N$, the set $\mathcal{X} \cap \mathbb Z^n$ denotes the feasible region and
$f: \mathcal{X} \cap \mathbb Z^n \to \mathbb{R}^p$ is a vector of~$p$ objective functions for some
$p \in \mathbb{N}$, $p\geq 1$. For $p=1$ we obtain an ordinary single-objective program, and for $p\geq 2$ we have multiple objectives. In particular, for $p=2$, a \emph{bi-objective integer program (BOIP)} is obtained. We assume that $\mathcal{X}\cap \mathbb Z^n$ is finite and that all objective functions take only rational values.  
The space of the $x$-variables, i.e., $\mathbb R^n$, is called the decision space and we usually call its elements solutions, and the space of $f(x)$, i.e., $\mathbb R^p$, is called the objective space or criterion space and we typically call its elements points. 
We refer the interested reader to the textbook by \citet{Ehrgott2005} for more details on MOIPs.

Since there are multiple objective functions, there is usually no single solution that is optimal with respect to all objective functions, and different concepts of optimality have been defined. In this work, we use the
concept of Pareto optimality, which is one of the most widely-used concepts.
A feasible solution $x \in \mathcal{X} \cap \mathbb Z^n$ is called efficient (or Pareto optimal) if there 
exists no other solution $x ' \in \mathcal{X} \cap \mathbb Z^n$ such that $f_q(x ') \leq f_q(x )$ 
for all $q \in \{1, \dots, p\}$ and for at least one $q \in \{1, \dots, p\}$ the inequality is strict.
If $x$ is efficient, $f(x)$ is called a nondominated point. A point that is not nondominated is called dominated. The set of all efficient
solutions is denoted by $\mathcal{X}_E$ and the set of all nondominated points is called nondominated set (or Pareto front) and is denoted by~$\mathcal{Y}_N$. 
Note that in the case of an MOIP, where the feasible region is a finite and discrete set, the nondominated set is also a finite, discrete set. Moreover, there can be multiple efficient solutions which give the same point in the nondominated set, i.e., there can exist $x, x' \in \mathcal{X} \cap \mathbb Z^n$, $x \neq x'$ with $f(x)=f(x')$. Thus, the usual goal in multi-objective optimization (considering Pareto optimality) is to obtain the nondominated set, i.e., to obtain one efficient solution for each point in the nondominated set.

There are various methods to solve MOIPs, which can be categorized into two types of approaches, namely decision space searches and criterion space searches. Decision space search algorithms explore the decision space by fixing variables and obtaining bound sets on the nondominated set $\mathcal{Y}_N$, similar to the single-objective \BB, see, e.g., \citet{przybylski_multi-objective_2017}. On the other hand, criterion space search algorithms systematically explore the criterion space by iteratively solving single-objective optimization problems like the $\varepsilon$-constraint method by \citet{Haimes1971}, which we explain in more detail in Section~\ref{sec:epsMethod}, or the two-phase method by \citet{ULUNGU1995}. 

\subsection{The multi-objective single-row facility location problem}\label{sec:mo-srflp}
Using the concepts defined above, we are now ready to give a formal definition of the main problem we consider in this paper, the \MOSRFLP.

\begin{definition}
\label{def:MO-SRFLP}
Let $p \in \mathbb{N}$, $p\geq 1$ and let $I = \{1,  \dots, n\}$ be a set of $n$ one-dimensional facilities with positive lengths 
$\ell_i \in \mathbb{R}_{>0}$ for each $i \in I$ and non-negative symmetric weights $c_{ij}^q \in \mathbb{R}_{\geq0}$ for each
pair of facilities $i,j \in I$ for each $q \in \{1, \dots, p\}$.
Let $\Pi_n$ be the set of all permutations $\pi$ of $\{1, \dots, n\}$ 
and let $D_{\pi}(i,j)$ be the sum of the lengths of the facilities between 
the facilities $i$ and $j$ in the permutation $\pi \in \Pi_n$ of the $n$ facilities of $I$, so 
$$\frac{\ell_i}{2} + D_{\pi}(i,j) + \frac{\ell_j}{2}$$
is the center-to-center distance between facilities $i$ and $j$ in $\pi$.

Then the \emph{multi-objective single-row facility layout problem (\MOSRFLP)} is defined as
\begin{equation*}
    \min_{\pi \in \Pi_n} f(\pi),  
\end{equation*}
where for each $q \in \{1, \dots, p\}$ the $q$-th objective function $f_q(\pi)$ of $f(\pi)$ is defined as 
\begin{equation*}
    f_q(\pi) =  \sum_{\substack{i,j \in I \\ i < j}} c_{ij}^q 
    \left( \frac{\ell_i}{2} + D_{\pi}(i,j) + \frac{\ell_j}{2} \right). 
\end{equation*}
The goal of the \MOSRFLP is to determine the nondominated set, i.e., to find one efficient solution for each point in the nondominated set.
 
\end{definition} 

As already mentioned in the introduction, for $p=1$ the \MOSRFLP is the classical \SRFLP.

\subsection{Mathematical formulations}\label{sec:mathForm}
In this section, we provide several formulations for the single-objective \SRFLP that we use to derive a formulation for the \MOSRFLP.
\subsubsection{Formulations for the single-objective \SRFLP}\label{sec:SRFLPForm}
\citet{anjos_semidefinite_2005} presented a formulation for
the \SRFLP using ordering variables $x = (x_{ij})_{i,j \in I, i < j}$, $x_{ij} \in \{-1, 1\}$ to represent permutations. In particular, for each pair of facilities $i,j \in I$ with $i<j$ they use the variable $x_{ij}$, where $x_{ij}=1$ holds if and only if the facility $i$ is positioned to the left of facility~$j$ in the permutation and $-1$ otherwise. 
Moreover, \citet{anjos_semidefinite_2005} define
\begin{equation}    
    K = \left(\sum_{\substack{i,j \in I \\ i < j}} \frac{c_{ij}}{2}\right)
    \left(\sum_{k\in I}\ell_k\right) \label{eq:const}
\end{equation}
and with that present the quadratic program
\begin{subequations}
\label{QP:SRFLP}
\begin{alignat}{3}
    \min \quad & K - \sum_{\substack{i,j \in I \\ i < j}} \frac{c_{ij}}{2}
    \left[\sum_{\substack{k \in I \\ k < i}}\ell_kx_{ki}x_{kj} - 
    \sum_{\substack{k \in I \\ i < k < j}}\ell_kx_{ik}x_{kj} \right.&&\left.+ 
    \sum_{\substack{k \in I \\ k > j}}\ell_kx_{ik}x_{jk}\right] \label{eq:SOQ-obj} \\
    \text{s.t.}\quad & x_{ij}x_{jk} - x_{ij}x_{ik} - x_{ik}x_{jk} = -1 
    && \forall i,j,k \in I, i < j < k \label{eq:SOQ-cycle}\\
    &x_{ij} \in \{-1, 1\} && \forall i,j \in I, i < j. 
    \label{eq:SOQ-binary}
\end{alignat}
\end{subequations}
as formulation for the \SRFLP. They proved that the so-called 
3-cycle constraints~\eqref{eq:SOQ-cycle} together with~\eqref{eq:SOQ-binary} suffice to exactly represent all possible permutations, because they imply $(x_{ij}+x_{jk})(x_{ij}-x_{ik}) = 0$ for all $i,j,k \in I$ with $i < j < k$ and thus model transitivity (if $j$ is between $i$ and $k$, so the first bracket is not zero, then the second bracket needs to be zero, so $i$ is either to the left or to the right of both $j$ and $k$).
They also showed that the objective function~\eqref{eq:SOQ-obj} coincides with the objective function of the \SRFLP. 

Based on the quadratic program~\eqref{QP:SRFLP}, \citet{anjos_semidefinite_2005} derived a formulation for the \SRFLP in the variable matrix $X \in \mathbb R^{\binom{n}{2}\times\binom{n}{2}}$ which allowed them to obtain a straightforward SDP relaxation. However, in our work, we use an extended SDP-based formulation presented by \citet{hungerlander_semidefinite_2013}, which uses both the variables 
$x \in \mathbb R^{\binom{n}{2}}$ and $X \in \mathbb R^{\binom{n}{2}\times\binom{n}{2}}$.
Towards this end, let
$\widetilde C = (\widetilde C_{ij,kh})_{i<j,k<h \in I} \in \mathbb{R}^{\dimension \times \dimension}$ with
\begin{align}
    \widetilde C_{ij,kh} &= 
    \begin{cases}
        -\dfrac{c_{j\fourindex}}{2}\ell_i & \text{if } i < j < \fourindex, k=i \\[4pt]
        \dfrac{c_{i\fourindex}}{2}\ell_j & \text{if } i<j<\fourindex, k=j \\[4pt]
        -\dfrac{c_{ik}}{2}\ell_j & \text{if } i<k<j, \fourindex=j \\[4pt]
        0 & \text{otherwise}, 
    \end{cases}\label{eq:costMatrix}
\end{align}
and for obtaining a symmetric matrix let $C = \dfrac{1}{2}{(\widetilde C + \widetilde C^T)}$,
then \citet{hungerlander_semidefinite_2013} formulated the \SRFLP as
\begin{subequations}
\label{SRFLP}
\begin{align}
    \quad \min \quad & K + \langle C, X\rangle \label{eq:SRFLP-Obj}\\
    \text{s.t.} \quad & X_{ij,jk} - X_{ij,ik} - X_{ik,jk} = -1 &\qquad& \forall i, j
    , k \in I, i < j < k \label{eq:3Cycle}\\
    & \diag(X) = e \label{eq:SRFLP-diag}\\
    & \begin{pmatrix}
    1 & x^T \\
    x & X
\end{pmatrix}\succeq 0 \label{eq:SRFLP-psd}\\
    & x_{ij} \in \{-1,1\} && \forall i, j \in I, i < j. \label{eq:x_integer}
\end{align}
\end{subequations}
This is not the only formulation for the \SRFLP using these variables, as the next result shows.

\begin{observation}\label{obs:rank-x_integer}
    Also~\eqref{SRFLP} when replacing~\eqref{eq:x_integer} with the constraint
    \begin{align}
    & \rank\begin{pmatrix}
    1 & x^T \\
    x & X
\end{pmatrix} = 1 \label{eq:SRFLP-rank}
    \end{align}
    is a formulation for the \SRFLP.
\end{observation}
\begin{proof} 
    We only need to show that~\eqref{eq:x_integer} is equivalent to~\eqref{eq:SRFLP-rank} with the other constraints of~\eqref{SRFLP}.
    Clearly~\eqref{eq:SRFLP-rank} implies $X = xx^T$, and together with~\eqref{eq:SRFLP-diag} this implies~\eqref{eq:x_integer}. To see the other direction, note that~\eqref{eq:x_integer} implies $x_{ij}^2 = 1$, so it holds that $\diag(X - xx^T)=0$. Furthermore, by Schur's complement~\eqref{eq:SRFLP-psd} yields $X-xx^T \succeq 0$, so $X = xx^T$ and therefore~\eqref{eq:SRFLP-rank} holds.
\end{proof}

We now make several remarks.
First, we want to mention that the objective function~\eqref{eq:SRFLP-Obj} and the constraint~\eqref{eq:3Cycle} correspond to the objective function~\eqref{eq:SOQ-obj} and the constraint~\eqref{eq:SOQ-cycle} in matrix notation, respectively.

Additionally, we want to point out that the mentioned formulation for \citet{anjos_semidefinite_2005} in only $X$ variables coincides with~\eqref{SRFLP} when replacing~\eqref{eq:x_integer} with~\eqref{eq:SRFLP-rank}, where~\eqref{eq:SRFLP-rank} and~\eqref{eq:SRFLP-psd} are adapted to $\rank(X) = 1$ and $X\succeq0$, respectively. This formulation is also mentioned in \citet{anjos_provably_2009}. 
However, the presence of the $x$-variables in~\eqref{SRFLP} allows us to use easier branching schemes, see Section~\ref{sec:binarybranch} for details. 

Furthermore, we observe that we obtain an SDP relaxation of the \SRFLP by dropping constraint~\eqref{eq:x_integer} from the above formulation~\eqref{SRFLP}. 
Also here, having the formulation with $x$-variables is beneficial. 
Indeed, it was already observed by \citet{hungerlander_computational_2013} 
that in general~\eqref{eq:SRFLP-psd} is a stronger constraint than $X\succeq0$, so when considering SDP relaxations, using~\eqref{eq:SRFLP-psd} gives the same or better bounds.

Finally, in \citet{anjos_provably_2009} the equations~\eqref{eq:3Cycle} have been aggregated by the index $k$ to obtain the constraints
\begin{align}
    \sum_{\substack{k \in I, \\ i \neq k \neq j}} X_{ij,jk} - X_{ij,ik} - X_{ik,jk} = -(n-2) && 
    \forall i, j \in I, i < j. \label{eq:agg3cycle}
\end{align}
Note that the constraints~\eqref{eq:agg3cycle} may include variables with indices that are not defined, e.g., $X_{23,31}$ when $i=2$, $j=3$ and $k=1$. 
However, the original ordering variables $x = (x_{ij})_{i,j \in I, i < j} \in \mathbb R^{\binom{n}{2}}$ can be extended in a canonical way to extended ordering variables $\bar{x} = (\bar{x}_{ij})_{i,j \in I} \in \mathbb R^{n(n-1)}$, where again $\bar{x}_{ij}=1$ holds if and only if facility $i$ is positioned to the left of facility~$j$ in the permutation and $-1$ otherwise. 
Then clearly $x_{ij} = \bar{x}_{ij} = -\bar{x}_{ji}$ holds for all $i,j \in I$ with $i<j$. This together with the intuition of $X = xx^T$ and the help of the function
\begin{align*}
    \tau(i,j) &= 
    \begin{cases}
        1 & \text{if } i < j \\
        -1 & \text{otherwise} 
    \end{cases}
\end{align*}
then implies that when using
\begin{align}\label{eq:transformationileqj}
    X_{ij,kh} = \tau(i,j) \tau(k,h) X_{\min\{i,j\}\max\{i,j\}, \min\{k,h\}\max\{k,h\}} && 
    \forall i, j,k,h \in I 
\end{align}
everything fits together and~\eqref{eq:agg3cycle} can be represented using only existing variables.

The constraints~\eqref{eq:agg3cycle} can also be used in formulation~\eqref{SRFLP} instead of~\eqref{eq:3Cycle}, as the next result shows.

\begin{observation}
   Also~\eqref{SRFLP} when replacing~\eqref{eq:3Cycle} with the constraints~\eqref{eq:agg3cycle} is a formulation for the \SRFLP.
\end{observation}
\begin{proof}
    We only need to show that~\eqref{eq:3Cycle} is equivalent to~\eqref{eq:agg3cycle} with the other constraints of~\eqref{SRFLP}.
    This can be done with analogous arguments as the ones in the proof of Theorem~3.1 in \citet{anjos_provably_2009}, when additionally using the arguments of Observation~\ref{obs:rank-x_integer} that show that $X \in \{-1,1\}^{\binom{n}{2} \times \binom{n}{2}}$ and the fact that~\eqref{eq:SRFLP-psd} implies that $X \succeq 0$.
\end{proof}

As a result, replacing~\eqref{eq:3Cycle} with~\eqref{eq:agg3cycle} when considering~\eqref{SRFLP} reduces the number of constraints to $\mathcal{O}(n^2)$. Unfortunately, this may lead to a weaker SDP relaxation. Preliminary computations showed that this reduction is still preferable for our algorithm. Thus, we 
use~\eqref{eq:agg3cycle} instead of~\eqref{eq:3Cycle} in the following.

\subsubsection{Formulation for the multi-objective \SRFLP}\label{sec:MO-SRFLPForm}

Next, we use the single-objective \SRFLP formulations described in the previous section to present a formulation for the \MOSRFLP. Given an instance of the \MOSRFLP, i.e., lengths $\ell_i \in \mathbb{R}_{>0}$ for each $i \in I$ and weights $c_{ij}^q \in \mathbb{R}_{\geq0}$ between all pairs of facilities $i,j \in I$ for each $q \in \{1, \dots,p\}$, let $C_q \in \mathbb{R}_{\geq0}^{\dimension \times \dimension}$ 
be the cost matrix $C$ constructed as described in~\eqref{eq:costMatrix} and let $K_q$ 
be the $K$ as defined in~\eqref{eq:const} for the weights $c_{ij} = c_{ij}^q$ for each $q \in \{1, \dots,p\}$.
We can then 
formulate the \MOSRFLP as 
\begin{subequations}
\label{eq:MO-SRFLP}
    \begin{align}
        \min        \quad   & (K_1+\langle C_1, X\rangle,\dots,K_p+\langle C_p, X\rangle) \\
        \text{s.t.} \quad  & \sum_{\substack{k \in I, \\ i \neq k \neq j}} X_{ij,jk} - X_{ij,ik} - X_{ik,jk} = -(n-2) &\qquad& \forall i, j \in I, i < j \label{eq:MO-SRFLP-3Cycle}\\
                            & \diag(X) = e \label{eq:MO-SRFLP-diag}\\
                            & \begin{pmatrix}
                                1 & x^T \\
                                x & X
                            \end{pmatrix}\succeq 0 \label{eq:MO-SRFLP-psd}\\
                             & x_{ij} \in \{-1,1\} && \forall i, j \in I, i < j. \label{eq:MO-SRFLP-xint} 
    \end{align} 
\end{subequations}
In the remainder of the paper we consider the \emph{bi-objective single-row facility location problem (\BOSRFLP)}, i.e.,  the \MOSRFLP~\eqref{eq:MO-SRFLP} for $p=2$, 
and we develop a solution algorithm to tackle the \BOSRFLP that utilizes the $\varepsilon$-constraint method and an SDP-based branch-and-bound algorithm.

\section{Solution algorithm}\label{sec:algorithm}
In this section, we present our solution algorithm to obtain the nondominated set of the \BOSRFLP. This solution algorithm is based on the $\varepsilon$-constraint method, which we detail in Section~\ref{sec:epsMethod}. We solve the single-objective problems that emerge within the $\varepsilon$-constraint method with a custom SDP-based \BB algorithm presented in Section~\ref{sec:SO-BB}. Afterwards, we discuss binary and non-binary branching strategies for the \BB in Section~\ref{sec:binarybranch} and Section~\ref{sec:nonBin}, respectively. Finally, we present a technique to reduce the dimension of the SDP relaxations that are solved within our \BB in Section~\ref{sec:reduction}. 

\subsection{The \texorpdfstring{$\varepsilon$-constraint method}{epsilon-constraint method}}\label{sec:epsMethod}
The $\varepsilon$-constraint method was introduced in \citet{Haimes1971} and is a popular approach for computing the nondominated set of MOIPs. The idea of this method is to solve a series of single-objective problems, the so-called $\varepsilon$-constraint problems. In these problems, one of the objective functions is kept as objective function, and the other ones are limited by a vector~$\varepsilon$ with additional constraints.

We note that even though the $\varepsilon$-constraint method works for all $p\geq2$, we restrict ourselves to the case of $p=2$, as the computational complexity of the $\varepsilon$-constraint method scales exponentially with the number of objective functions $p$, see, e.g. \citet{LAUMANNS2006} for details.
In particular, for BOIPs and a given $\varepsilon \in \mathbb{R}$, we obtain the $\varepsilon$-constraint problem 
\begin{align*} \label{eq:epsProblem}
   \mytag{P(\varepsilon)} \quad  \min_{x \in \mathcal{X} \cap \mathbb Z^n} \quad &f_1 (x) \\
    \text{s.t.} \quad &f_2(x) \leq \varepsilon. 
\end{align*} 
Note, that one could also use $f_2$ as objective function and include $f_1$ into the constraints. 
The $\varepsilon$-constraint method solves these problems in an iterative fashion for different values of $\varepsilon$ in order to compute the nondominated set of a BOIP. The method is outlined in Algorithm~\ref{alg:epsMethod}.

As an input, it takes the two objective functions $f_1$ and $f_2$, the set $\mathcal{X}$ and a parameter $\delta$, for which 
\begin{equation} 
\label{eq:delta}
0 < \delta \leq \inf_{\substack{x,x' \in \mathcal{X} \cap \mathbb Z^n:\\ f_2(x) \neq f_2(x')}} |f_2(x) - f_2(x')|
\end{equation}
needs to hold, so $\delta$ is a lower bound on the difference of the second objective function values for any two feasible solutions, for which these values do not coincide. Note that such a value of $\delta$ can always be obtained because $\mathcal{X} \cap \mathbb Z^n$ is finite, hence the right-hand side of~\eqref{eq:delta} is positive or $\infty$ if all feasible solutions in $\mathcal{X} \cap \mathbb Z^n$ have the same second objective function value.
It is easy to see that one efficient solution $x^*$ associated with each point $(f_1(x^*), f_2(x^*))$ in the nondominated set is also obtained, but for ease of presentation we do not explicitly mention storing these solutions in Algorithm~\ref{alg:epsMethod}.

\begin{algorithm}[h!tb]
    \caption{\epsConstraint method for BOIPs}\label{alg:epsMethod}
\begin{algorithmic}[1]
    \Require $f_1$, $f_2$, $\mathcal{X}$, $\delta$
    \State $\mathcal{Y}_{N'} \gets \emptyset$, $\varepsilon \gets \infty$
    \State $(feas, x^*)$ $\gets $ \texttt{solve}$\myrefmath{P(\varepsilon)}$
    \While{$feas$}
        \State $\mathcal{Y}_{N'} \gets \mathcal{Y}_{N'} \cup \{(f_1(x^*), f_2(x^*))\}$
        \State $\varepsilon \gets f_2(x^*) - \delta$
        \State $(feas, x^*)$ $\gets $ \texttt{solve}$\myrefmath{P(\varepsilon)}$
    \EndWhile
    \State $\mathcal{Y}_N \gets$ \texttt{filterDominated}$(\mathcal Y_{N'})$ \label{alg:weakline}
    \State \Return $\mathcal{Y}_N$
\end{algorithmic}
\end{algorithm}

In this algorithm, \texttt{solve}$\myrefmath{P(\varepsilon)}$ denotes a black-box algorithm capable of solving $\myrefmath{P(\varepsilon)}$ which returns a tuple $(feas,x^*)$, where $feas$ is a flag which takes the value $true$ if $\myrefmath{P(\varepsilon)}$ is feasible, and in this case~$x^*$ contains an optimal solution to $\myrefmath{P(\varepsilon)}$. If $\myrefmath{P(\varepsilon)}$ is infeasible, the flag $feas$ contains the value $false$ and $x^*$ is not defined.

Moreover, the function \texttt{filterDominated}$(\mathcal Y_{N'})$ gets called at the end as a post-processing step. The reason for this is that otherwise 
the $\varepsilon$-constraint method does not produce only nondominated points. 
Thus, the function \texttt{filterDominated}$(\mathcal Y_{N'})$ partitions $\mathcal Y_{N'}$ into sets of points with the same value in the first objective function. For each such set, it removes all points except the one with the smallest value in the second objective function. 

We note that there are other methods to handle this issue, e.g., the lexicographic $\varepsilon$-constraint method described by \citet[Chapter 6]{Cohen1978} or a hybridization of the weighted sum and the $\varepsilon$-constraint method introduced by \citet{Neumayer1994}. While these methods are guaranteed to find only nondominated points, there are drawbacks as well, e.g., for the lexicographic $\varepsilon$-constraint method, at each iteration two problems have to be solved, and in the hybrid approach, weights have to be carefully determined.
As a result, we have decided to perform the described post-processing step.

The $\varepsilon$-constraint problem of the \BOSRFLP, i.e., the problem $\myrefmath{P(\varepsilon)}$ to be solved within one iteration of the $\varepsilon$-constraint method, is given by
\begin{subequations}
    \begin{align}
        \mytag{BOSRFLP(\varepsilon)} \quad \min \quad &K_1 + \langle C_1,X \rangle \label{eq:BO-obj}\\
        \text{s.t.} \quad & K_2 + \langle C_2, X \rangle \leq \varepsilon \label{eq:BO-eps} \\ 
        & \eqref{eq:MO-SRFLP-3Cycle}-\eqref{eq:MO-SRFLP-xint},
        \label{eq:BO-rest}
    \end{align}
\end{subequations}
and we solve the problem $\myrefmath{BOSRFLP(\varepsilon)}$ for fixed values of $\varepsilon$ using an SDP-based \BB algorithm. 

Before we give the details of how we do that, we quickly mention an alternative for $\myrefmath{P(\varepsilon)}$, which we will later on use in our computational study in Section~\ref{sec:comp-results} to evaluate the quality of our solution approach. In order to compare it with off-the-shelf solvers, we consider the \SRFLP ILP-formulation by \citet{amaral_new_2009} using betweenness variables~$b = (b_{ijk})_{\substack{{i,j,k \in I, i < j} \\i \neq k \neq j}} \in \mathbb R^{(n-2)\binom{n}{2}}$ to represent permutations. In particular, for every triplet $i,j,k \in I$ of facilities with $i < j$ and $i \neq k \neq j$, this formulation uses the variable $b_{ijk} \in \{0,1\}$, where $b_{ijk}=1$ if and only if facility $k$ is positioned between the facilities~$i$ and~$j$ in the permutation and $b_{ijk}=0$ otherwise. Using this notation and the formulation by~\citet{amaral_new_2009}, we obtain
\begin{subequations} \label{eq:BOSRFLPe-IP}
    \begin{align}
        \min \quad &\sum_{\substack{i,j \in I \\ i<j}} c_{ij}^1 \sum_{k \in I\setminus\{i,j\}} \ell_kb_{ijk} + \sum_{\substack{i,j \in I\\i<j}}c_{ij}^1 \frac{\ell_i + \ell_j}{2} \\
        \text{s.t. }\quad &\sum_{\substack{i,j \in I \\ i<j}} c_{ij}^2 \sum_{k \in I\setminus\{i,j\}} \ell_kb_{ijk} + \sum_{\substack{i,j \in I\\i<j}}c_{ij}^2 \frac{\ell_i + \ell_j}{2} \leq \varepsilon \\
         &b_{ijk}+b_{ikj}+b_{jki} = 1 && \forall i, j, k\in I, i<j<k \\
         &b_{ijh}+b_{jkh}+b_{ikh} \leq 2 && \forall i, j, k, h \in I, i<j<k<h \\
         &-b_{ijh}+b_{jkh}+b_{ikh} \geq 0 && \forall i, j, k, h \in I, |\{i, j, k, h\}|=4, i<j<k \\
         &b_{ijh}-b_{jkh}+b_{ikh} \geq 0 && \forall i, j, k, h \in I, |\{i, j, k, h\}|=4, i<j<k \\
         &b_{ijh}+b_{jkh}-b_{ikh} \geq 0 && \forall i, j, k, h \in I, |\{i, j, k, h\}|=4, i<j<k \\
         &b_{ijk} \in\{0,1\} && \forall i, j, k \in I, |\{i, j, k\}|=3 , i<j
    \end{align}
\end{subequations}
as an ILP formulation for the $\varepsilon$-constraint problem $\myrefmath{P(\varepsilon)}$ of the \BOSRFLP.

\subsection{Single-objective \BB}\label{sec:SO-BB}
We now provide a detailed explanation of our custom SDP-based \BB algorithm to solve $\myrefmath{BOSRFLP(\varepsilon)}$ for a fixed value of $\varepsilon$.
We note that for $\varepsilon=\infty$ this is an algorithm for solving the \SRFLP to optimality.
By dropping the constraints~\eqref{eq:MO-SRFLP-xint} we obtain an SDP relaxation of $\myrefmath{BOSRFLP(\varepsilon)}$. 

If (at least) one of the variables $x_{ij}$ does not take a value in $\{-1,1\}$ in the current optimal solution of the SDP relaxation at a node in the \BB tree, in our standard, i.e., binary branching-scheme (with multiple variants to select the branching variable, see Section~\ref{sec:binarybranch} for details), we pick one of these variables and branch on it to create two new child nodes (one node where the variable is fixed to $-1$, and another node where it is fixed to $1$). 

More formally, let $I^2=\{ ij: i,j \in I, i<j\}$, i.e., $I^2$ is the index-set of the $x$-variables. 
For a given node in the \BB tree, the sets $I^-\subseteq I^2$ and $I^+\subseteq I^2$ indicate which variables are already fixed to $-1$ and $1$ at this node, respectively.
Then the SDP relaxation which gets solved at this node is
\begin{subequations}
\label{eq:SDP_eps_node}
    \begin{align}
        \mytag{SDP(\varepsilon,I^-,I^+)} \quad z = \min \quad &K_1 +\langle C_1,X \rangle \label{eq:BOR-obj}\\
        \text{s.t.} \quad & K_2 + \langle C_2, X \rangle \leq \varepsilon \label{eq:BOR-eps} \\ 
        & \eqref{eq:MO-SRFLP-3Cycle}-\eqref{eq:MO-SRFLP-psd} \\
        &x_{ij}=-1 &\quad&\forall ij \in I^- \label{eq:BOR-down} \\
        &x_{ij}=1 &&\forall ij \in I^+ \label{eq:BOR-rest}.
    \end{align}
\end{subequations}

\begin{algorithm}[tb]
    \caption{SDP-based \BB algorithm to solve $BOSRFLP(\varepsilon)$ for fixed $\varepsilon$}\label{alg:BnB}
\begin{algorithmic}[1]
    \Require $\myrefmath{BOSRFLP(\varepsilon)}$
    \State $(z^{UB},x^{UB},X^{UB}) \gets (\infty,\emptyset,\emptyset)$ \Comment{initialize empty incumbent} 
  \State $\mathcal{Q} \gets \texttt{addNode}( - \infty, \emptyset,\emptyset)$ \Comment{initialize the tree} \label{line:init}
    \While{$\mathcal{Q}\neq \emptyset$}
        \State $(z^*, I^-,I^+) \gets \texttt{top}(\mathcal Q)$ \label{line:pickNode}
        \Comment{consider node in $\mathcal{Q}$ with smallest $z^*$}
        \If{$z^* \geq z^{UB}$} \Comment{prune by bound} \label{line:bound1}
            \State \textbf{break} \label{line:bound3}
            \Comment{prune by bound all remaining nodes in $\mathcal{Q}$}
        \EndIf
        \State $(feas,z^*,x^*,X^*)\gets \texttt{solve}(\myrefmath{SDP(\varepsilon,I^-,I^+)})$ \Comment{solve SDP relaxation to optimality}\label{line:solve}
        \If{$feas=false$} \Comment{prune by infeasibility}
            \State \textbf{continue}
        \EndIf
        \If{$x^*\in\{-1,1\}^{\binom{n}{2}}$} \Comment{prune by integrality} \label{line:integrality}
            \If{$z^*<z^{UB}$} \Comment{update the incumbent}
                \State $(z^{UB},x^{UB},X^{UB}) \gets (z^*,x^*,X^*)$
            \EndIf
            \State \textbf{continue}
        \EndIf
        \State $(\tilde z,\tilde x, \tilde X) \gets \texttt{heuristic}(X^*)$ \Comment{see Section~\ref{sec:heur}}
        \If{$\tilde z < z^{UB}$} \Comment{update the incumbent with the heuristic solution}
            \State $(z^{UB},x^{UB},X^{UB}) \gets (\tilde z,\tilde x, \tilde X)$
        \EndIf
        \If{$z^* \geq z^{UB}$} \Comment{prune by bound} \label{line:bound2}
            \State \textbf{continue}
        \EndIf
            \State $ij \gets \texttt{branchingIndex}(x^*,I^-,I^+)$ \Comment{see Section~\ref{sec:binarybranch}} \label{line:branch1}
            \State $\mathcal Q \gets \texttt{addNode}(z^*,I^- \cup \{ij\},I^+)$ \Comment{node where $x_{ij}$ is fixed to $-1$} \label{line:branch2}
            \State $\mathcal Q \gets \texttt{addNode}(z^*,I^-,I^+\cup \{ij\})$ \Comment{node where $x_{ij}$ is fixed to $1$} \label{line:branch3}
    \EndWhile
    \State \Return $(z^{UB},x^{UB},X^{UB})$
\end{algorithmic}
\end{algorithm}
Algorithm~\ref{alg:BnB} shows how we implemented our \BB algorithm to solve $\myrefmath{BOSRFLP(\varepsilon)}$.
A node in our \BB tree is determined by 
$(z,I^-,I^+)$, i.e., a priority $z$ and the variable fixings $I^-$ and $I^+$. The tree is implemented using a priority queue $\mathcal Q$, where the priorities $z$ of the nodes in $\mathcal Q$ are given by the objective function value of the relaxation of the parent of the node. We select the node with minimal priority to be considered next, i.e., we use a worst bound strategy for node selection in the function \texttt{top}$(\mathcal Q)$, which returns and removes the element with minimal priority from $\mathcal Q$ (ties are broken arbitrarily).
We initialize $\mathcal Q$ with \texttt{addNode}$(-\infty, \emptyset,\emptyset)$, which adds a root-node with priority minus infinity and no variable fixings.

We then check if the objective function value of the relaxation of the parent $z^*$ is smaller than the objective function value of the current incumbent $z^{UB}$, if not, we stop processing the node (pruned by bound). If yes, we then solve the current SDP relaxation with \texttt{solve}$(\myrefmath{SDP(\varepsilon,I^-,I^+)})$. This function returns a tuple with four entries, where the first one is a flag $feas$, which indicates if the current SDP relaxation is feasible or not. 
If it is not feasible, the remaining entries of the tuple are undefined, we stop processing the node (pruned by infeasibility) and continue with the next node. 
If the SDP relaxation is feasible, then the remaining entries of the tuple contain the optimal objective function value $z^*$ and an optimal solution ($x^*,X^*$) of the SDP relaxation. In this case, we first check if ($x^*,X^*$) is also feasible for $\myrefmath{BOSRFLP(\varepsilon)}$, i.e., if $x^* \in \{-1,1\}^{\binom{n}{2}}$. If yes, we check if we can update the incumbent, and then stop processing the node (pruned by integrality). If not, we first apply a heuristic (described in more detail in Section~\ref{sec:heur}) to try to obtain an improved incumbent. Then, we check if the optimal objective function value of the relaxation $z^*$ is smaller than the objective function value of the current incumbent $z^{UB}$, if not, we stop processing the node (pruned by bound). If yes, we have to branch. 

To do so, the function \texttt{branchingIndex}$(x^*,I^-,I^+)$ returns the index of the variable on which we do the branching (explained in detail in Section~\ref{sec:binarybranch}). 
Then, we create two new nodes, one where the variable with the indicated index is fixed to $-1$, and a second one where the variable gets fixed to $1$. Both nodes are added to $\mathcal Q$ with priority $z^*$ with $\texttt{addNode()}$. This concludes the processing of a node. Algorithm~\ref{alg:BnB} terminates as soon as all nodes are processed.

\subsubsection{Primal heuristic}\label{sec:heur}
To enhance our basic \BB, within \texttt{heuristic}($X^*$) we have implemented a primal heuristic which is driven by the $X^*$-part of an optimal solution of the SDP relaxation in a node, following the single-objective heuristic introduced by \citet{anjos_semidefinite_2005} and further improved by \citet{ANJOS2006}.
For $ij \in I^2$, the heuristic calculates a score $w^{ij}_k$ for each facility $k\in I$ based on the values of $X^*$ in row~$ij$ as 
\begin{equation}
    \omega_k^{ij} = \frac{1}{2}\left(n + 1 + \sum_{\substack{\fourindex \in I\\\fourindex < k}} X^*_{ij,\fourindex k} - \sum_{\substack{\fourindex \in I\\\fourindex > k}} X^*_{ij,k\fourindex} \right). \notag
\end{equation}
The facilities are then sorted in increasing order according to this score. This sorting gives a permutation, which corresponds to a vector $\hat x \in\{-1,1\}^{\binom{n}{2}}$. If the solution $(\hat x,\hat X=\hat x \hat x^T)$ satisfies~\eqref{eq:BO-eps}, then it is feasible for $\myrefmath{BOSRFLP(\varepsilon)}$. In our heuristic, we apply the above procedure to all rows $ij$ of~$X^*$, and return the best feasible solution obtained (if any). All feasible solutions obtained in these procedures are used to fill the solution pool as described later in Section~\ref{sec:pooling}.
\subsection{Binary branching}\label{sec:binarybranch}
In this section, we first describe binary branching strategies to determine the branching index. Then, we detail a method that potentially allows us to fix additional variables directly after the branching decision, possibly reducing the number of \BB nodes that we need to explore. 

\subsubsection{Branching strategies}\label{sec:binStrat}
 We consider five different branching strategies for choosing a branching index within the function $\texttt{branchingIndex}(x^*,I^-,I^+)$:
\begin{itemize}
    \item The \texttt{mostInfeasible} branching strategy chooses an index $ij \in \argmin_{i'j'\in I^2 \setminus (I^- \cup I^+)} |x^*_{i'j'}|$, i.e., an index $ij$ with $x^*_{ij}$ closest to zero. 
    This corresponds to most fractional branching in a classical \BB for problems with binary variables.

    \item The \texttt{lengthWeighted} branching strategy is a modification of the \texttt{mostInfeasible} branching strategy that is weighted by the lengths of the facilities, thus giving higher priority to indices corresponding to facilities with longer lengths. 
    In particular, it chooses an index
    \[
        ij \in \argmax_{i'j' \in I^2 \setminus (I^- \cup I^+)} (\ell_{i'} + \ell_{j'}) (1-|x^*_{i'j'}|).
    \]

    \item The \texttt{connected} branching strategy prioritizes indices whose corresponding facilities are already involved in the fixings $I^-$ and $I^+$ at the current node, i.e., facilities in 
    \[
        \hat{I}:=\{i \in I: \exists i'j' \in I^-\cup I^+ \text{ such that } i \in \{i',j'\}\}.
    \]
    The strategy proceeds in three steps:
    \begin{enumerate}
        \item Consider all indices $ij \in I^2$ with $|x^*_{ij}|<1$ for which both facilities $i$ and $j$ are involved in a fixing, i.e., $i,j \in \hat{I}$. Among these indices, choose one where the value of $x^*_{ij}$ is closest to zero. If no such index exists, go to the next step.
        \item Consider all indices $ij \in I^2$ with $|x^*_{ij}|<1$ for which one of the facilities $i$ and $j$ is involved in a fixing, i.e., $i \in \hat{I}$ or $j \in \hat{I}$. Among these indices, choose one where the value of $x^*_{ij}$ is closest to zero. If no such index exists, go to the next step.
         \item Fall back to the \texttt{mostInfeasible} branching strategy.
    \end{enumerate}
    
    \item The \texttt{facilityFocused} branching strategy also prioritizes indices whose corresponding facilities are already involved in the fixings $I^-$ and $I^+$ at the current node, however, unlike the \texttt{connected} branching strategy, it also considers the number of times a facility is involved in the fixings. This narrows the focus on a small subset of indices. To avoid focusing too much on the same facilities, we disregard facilities that are more than $n/2$ times involved in fixings. More formally, let
    \[
        \hat{I}^2_i := \{i'j' \in I^- \cup I^+: i \in \{i', j'\}\}
    \]
    be the set of indices involved in fixings containing facility $i$. The strategy proceeds in two steps:
    \begin{enumerate}
    \item Disregard any facility $i$ with $|\hat{I}^2_i|>n/2$. Sort the remaining facilities $i$ in a non-increasing fashion according to $|\hat{I}^2_i|$ and let $i_1, i_2, i_3, \ldots$ denote the ordered facilities after the sorting. Then, start with $i_1$ and $i_2$, create the corresponding index ($i_1i_2$ or $i_2i_1$, depending on whether $i_1<i_2$ or not), without loss of generality $i_1i_2$. If $x^*_{i_1i_2} \notin \{-1,1\}$, choose this index. If $x^*_{i_1i_2}\in \{-1,1\}$, check the index associated with the pair $i_1$ and $i_3$. Continued (with first trying $i_1$ with all other indices, then $i_2$ with all other indices, and so on) until either finding an index, for which the corresponding variable in $x^*$ is not in $\{-1,1\}$, and choosing this index, or until all potential pairs are checked. In the latter case, go to the next step.

    \item Fall back to the \texttt{mostInfeasible} branching strategy.
    \end{enumerate}     

    \item The \texttt{disconnected} branching strategy aims to involve all facilities in some fixing early on, thus opposing the \texttt{connected} and \texttt{facilityFocused} branching strategies. It first considers all indices~$ij$, for which the facilities $i$ and $j$ are both not involved in a fixing at the current node, i.e., all indices~$ij$ with $|x^*_{ij}|<1$ and $i,j \in I \setminus \hat{I}$. Among these indices, it chooses one where the value of $x^*_{ij}$ is closest to zero. If no such index exists, the strategy proceeds with steps two and three of the \texttt{connected} branching strategy.
    
\end{itemize}

Finally, note that the objective function $f(\pi)$ of a permutation $\pi \in \Pi_n$ as defined in Definition~\ref{def:MO-SRFLP} is symmetric with respect to considering a permutation from left to right or from right to left, i.e., two permutations that differ only in the orientation have the same objective function value. Therefore, to avoid exploration of redundant nodes, for all binary branching strategies we can initially fix $x_{ij}$ to either $1$ or $-1$ for one $ij \in I^2$ for breaking symmetry. Specifically, in our implementation for all binary branchings we initialize $I^- = \emptyset$ and $I^+ = \{12\}$, i.e., we set $x_{12}=1$.

\subsubsection{Sequential fixing}\label{sec:sequential}
After choosing a branching index within \texttt{branchingIndex()} that induces one new fixing within each of the two subsequent branches of the \BB tree, we check if these new fixings within $I^-$ and $I^+$ have further implications. In particular, the constraints~\eqref{eq:MO-SRFLP-3Cycle} of our formulation stem from the 3-cycle-constraints~\eqref{eq:SOQ-cycle}, which therefore need to be satisfied by any feasible solution $x$. 
As a result, whenever two of the three variables $x_{ij}$, $x_{ik}$ and $x_{jk}$ for some $i<j<k$ occurring in~\eqref{eq:SOQ-cycle} are already fixed to~$1$ or~$-1$, the value of the third variable may be implied and thus can also be fixed in a \emph{sequential fixing}. 

More formally, let $i'j'$ be the branching index determined by \texttt{branchingIndex()} to be added to $I^-$ and $I^+$ to create the branches of the \BB tree. 
Then, to perform the sequential fixing, for each of the two new branches of the \BB tree and for every $k'\in I \setminus \{i',j'\}$, consider the reordering $i$, $j$, $k$ of $i'$, $j'$, $k'$ with $\{i,j,k\} = \{i',j',k'\}$ and $i<j<k$ and
\begin{itemize}[noitemsep]
\item if $ij \in I^+$ and $jk \in I^+$, then add $ik$ to $I^+$,
\item if $ij \in I^+$ and $ik \in I^-$, then add $jk$ to $I^-$,
\item if $ij \in I^-$ and $jk \in I^-$, then add $ik$ to $I^-$,
\item if $ij \in I^-$ and $ik \in I^+$, then add $jk$ to $I^+$,
\item if $ik \in I^+$ and $jk \in I^-$, then add $ij$ to $I^+$, and
\item if $ik \in I^-$ and $jk \in I^+$, then add $ij$ to $I^-$.
\end{itemize}

\subsection{Non-binary branching}\label{sec:nonBin}

The branching within the function \texttt{branchingIndex()} described so far is a binary branching in the sense that each branching decision results in two new nodes and thus two new branches in the \BB tree. An alternative to this binary branching is to construct all feasible solution to $\myrefmath{BOSRFLP(\varepsilon)}$, i.e., all permutations of $I=\{1,2,\dots,n\}$, by iteratively placing at each position in $\{1,2,\dots,n\}$ one facility in~$I$. We note that such a branching strategy is already briefly mentioned in \citet{Schwiddessen2020}.

In this strategy, we construct the permutations starting from the positions at the outside, alternating between locating a facility at the leftmost and rightmost free position. 
In particular, in the first branching step, a node is created for placing at position 1 (the leftmost position) the facility~$i$ for each facility $i \in I$.
In the next branching step within each of the nodes, a new node is created for placing at position $n$ (the rightmost position) the facility $j$ for every remaining facility $j\in I \setminus \{i\}$. This process is repeated for all remaining positions by placing all remaining facilities at position 2, then position $n-1$, and so on, until the permutation is complete. Clearly, in this way we enumerate all permutations.
    
\begin{subalgorithms}
\begin{algorithm}[htb]
    \caption{Non-binary branching adaptation that replaces line~\ref{line:init} in Algorithm~\ref{alg:BnB}}\label{alg:nonBinaryInit}
\begin{algorithmic}[1]
  \State $\mathcal{Q} \gets \texttt{addNode}(-\infty, \emptyset,\emptyset,\emptyset,true)$ \Comment{initialize the tree}
\end{algorithmic}
\end{algorithm}

\begin{algorithm}[htb]
    \caption{Non-binary branching adaptation that replaces line~\ref{line:pickNode} in Algorithm~\ref{alg:BnB}}\label{alg:nonBinaryUpdate}
\begin{algorithmic}[1]
    \State $(z^*, I^-,I^+,\bar{I},left) \gets \texttt{top}(\mathcal Q)$ 
    \Comment{consider node in $\mathcal{Q}$ with smallest $z^*$}
\end{algorithmic}
\end{algorithm}

\begin{algorithm}[htb]
    \caption{Non-binary branching adaptation that replaces lines~\ref{line:branch1}-\ref{line:branch3} in Algorithm~\ref{alg:BnB}}\label{alg:nonBinaryMain}
\begin{algorithmic}[1]
    \For{$i \in I \setminus \bar{I}$}
        \If{$left=true$} 
        \State $I^{++} = \{ij: j \in I \setminus\bar{I}, i < j\},\, I^{--} = \{ji: j \in I \setminus\bar{I}, i > j\}$
        \Else
        \State $I^{--} = \{ij: j \in I \setminus\bar{I}, i < j\},\, I^{++} = \{ji: j \in I \setminus\bar{I}, i > j\}$
        \EndIf
    \State $\mathcal{Q} \gets $ \texttt{addNode}$(z^*, I^-\cup I^{--}, I^+\cup I^{++}, \bar{I} \cup \{i\}, \textbf{not } left)$
    \EndFor
\end{algorithmic}
\end{algorithm}
\end{subalgorithms}

Our non-binary branching strategy is formally described by Algorithms~\ref{alg:nonBinaryInit},~\ref{alg:nonBinaryUpdate} and~\ref{alg:nonBinaryMain}, which replace the lines~\ref{line:init},~\ref{line:pickNode} and~\ref{line:branch1}-\ref{line:branch3} in Algorithm~\ref{alg:BnB}, respectively. 
Our strategy requires a redefinition of the elements in the priority queue~$\mathcal{Q}$, i.e., the nodes of the \BB tree. A node now consists of a tuple with five instead of three elements, where the two additional elements are the set of facilities that are already fixed $\bar{I} \subseteq I$ and a boolean flag \emph{left}, indicating if the next facility has to be placed at the leftmost (\emph{left}=\emph{true}) or rightmost (\emph{left}=\emph{false}) possible position. 

At the initialization in Algorithm~\ref{alg:nonBinaryInit} we set $I^- = I^+ = \bar{I} = \emptyset$ and $left=true$ because no facilities are fixed and we start to place facilities from left. The exploring of a new node in the \BB tree is then started in Algorithm~\ref{alg:nonBinaryUpdate}. When actually performing the branching in Algorithm~\ref{alg:nonBinaryMain}, for each facility $i \in I\setminus\bar{I}$ that is not already fixed, we determine the fixings to $-1$ and $1$ that are induced by placing $i$ at the leftmost or rightmost possible position within the two sets $I^{--}, I^{++} \subseteq I^2$.
Then, a node is added to the priority queue $\mathcal{Q}$ with the updated sets and the inverted flag \emph{left}.

The symmetry of the objective function mentioned in Section~\ref{sec:binStrat} can be broken for the non-binary branching strategy as well in the following way.
To avoid the exploration of redundant nodes within our \BB tree, we consider only those permutations, where the (index of the) facility on the leftmost position is smaller than the (index of the) facility of the rightmost position, i.e., $\pi^{-1}(1) < \pi^{-1}(n)$. 

Note that the sequential fixing described in Section~\ref{sec:sequential} does not have any effect on our non-binary branching strategy, because no additional fixings can occur within the sequential fixing by construction.

\subsection{Node relaxation reduction}\label{sec:reduction}
In principle, it is possible to implement the branching with one additional equality constraint~\eqref{eq:BOR-down} or~\eqref{eq:BOR-rest} per fixed variable, which results in solving $\myrefmath{SDP(\varepsilon,I^-,I^+)}$ as described above. 
However, it is well-known that the performance of interior-point based SDP solvers (a type of solver which we also use in our computations) deteriorates with an increasing number of constraints, see, e.g.,~\citet{BORCHERS2007}. 

It is now our goal to describe a different way of implementing the branching using a \emph{reduction} to overcome this obstacle.
It is based on the following result, which shows that the branching decisions~\eqref{eq:BOR-down} and~\eqref{eq:BOR-rest} have further implications for \myref{SDP($\varepsilon,I^-,I^+$)}.
\begin{observation}
\label{obs:red}
In \myref{SDP($\varepsilon,I^-,I^+$)}, any constraint~\eqref{eq:BOR-down} and~\eqref{eq:BOR-rest}, i.e., any constraint of the form $x_{ij} = \sigma$ for some $ij \in I^2$ with $\sigma \in \{-1,1\}$, implies 
\begin{align*}
    X_{ij,kh} = X_{kh,ij} = \sigma x_{kh} \quad \forall kh \in I^2.
\end{align*} 
Thus, the column of $X$ indexed by $ij$ equals $\sigma x$ and the row of $X$ indexed by $ij$ equals $\sigma x^T$.
\end{observation}
\begin{proof}
Due to~\eqref{eq:MO-SRFLP-diag} and the SDP-constraint~\eqref{eq:MO-SRFLP-psd}, this result follows 
by using the non-negativity of the determinant of the $3\times3$ submatrix of 
$\begin{pmatrix}
1 & x^T \\
x & X
\end{pmatrix}$ 
that corresponds to the columns and rows indexed by~$1$, $ij$ and $kh$. 
\end{proof}

We exploit Observation~\ref{obs:red} by using the variable fixings within $I^-$ and $I^+$ caused by the branching and the sequential fixing described in Section~\ref{sec:sequential} to obtain a reduced SDP equivalent to $\myrefmath{SDP(\varepsilon,I^-,I^+)}$ with a smaller number of variables and constraints, which directly incorporates the fixings. 
In particular, the following holds.

\begin{proposition}    \label{prop:red}
Consider a node in the B\&B tree with given $\varepsilon$, $I^-$ and $I^+$.
Let $I^2_F = I^- \cup I^+$ and $I^2_R = I^2\setminus I^2_F$ be the fixed and non-fixed indices of $I^2$, respectively. 
Let $\sigma^F \in \{-1,1\}^{|I^2_F|}$ be the vector of all fixings in $I^2_F$. 

Then $\myrefmath{SDP(\varepsilon,I^-,I^+)}$ is equivalent to 
\begin{subequations}
    \begin{align}
        \mytag{SDP^R(\varepsilon,I^-,I^+)} \quad z = \min \quad &K_1^R +\langle C_1^R,X^R \rangle + \langle c_1^R, x^R \rangle \label{eq:SDPR-obj}\\
        \text{s.t.} \quad & K_2^R + \langle C_2^R, X^R \rangle + \langle c_2^R, x^R \rangle \leq \varepsilon \label{eq:SDPR-eps} \\ 
         & \sum_{\substack{k \in I, \\ i \neq k \neq j}} \widetilde{X}_{ij,jk} - \widetilde{X}_{ij,ik} - \widetilde{X}_{ik,jk} = -(n-2) &\quad&  \label{eq:SDPR-3Cycle}
         \forall i, j \in I, i < j \\
                            & \diag(X^R) = e \label{eq:SDPR-diag} \\
                            & \begin{pmatrix}
                                1 & (x^R)^T \\
                                x^R & X^R
                            \end{pmatrix}\succeq 0 \label{eq:SDPR-psd},
    \end{align}
\end{subequations}
where $x^R \in \mathbb R^{|I^2_R|}$, $X^R \in \mathbb R^{|I^2_R|\times|I^2_R|}$, 

\begin{align} \label{eq:fixingImplications}
\widetilde{X}_{ij,kh} = 
    \begin{cases}
        X^R_{ij,kh} & \text{if } ij,kh \in I^2_R \\
        \sigma^F_{ij}x^R_{kh} & \text{if } ij \in I^2_F, kh \in I^2_R\\        
        \sigma^F_{kh}x^R_{ij} & \text{if } kh \in I^2_F, ij \in I^2_R\\
        \sigma^F_{ij}\sigma^F_{kh} & \text{if } ij,kh \in I^2_F,
    \end{cases}
\end{align}
for all $ij, kh \in I^2$ (and thus for all $i,j,k,h \in I$ with $i\neq j$ and $k \neq h$ using~\eqref{eq:transformationileqj} to obtain $i<j$ and $k<h$), 
and $K_q^R = K_q + \langle C_{q_{I^2_F, I^2_F}}, \sigma^F(\sigma^F)^T \rangle$, $C_q^R = C_{q_{I^2_R, I^2_R}}$ and $c_q^R = 2(C_{q_{I^2_F,I^2_R}})^T \sigma^F$ for all $q\in \{1,2\}$. 
\end{proposition}

\begin{proof}
    The variable $x$ in $\myrefmath{SDP(\varepsilon,I^-,I^+)}$ can be split into $x^R \in \mathbb R^{|I^2_R|}$ and $x^F \in \mathbb R^{|I^2_F|}$ with $x^F$ representing all entries of $x$ that correspond to fixings in $I^2_F$ as 
    $x=\begin{pmatrix} x^R \\ x^F \end{pmatrix}$. 
    Similarly, $X$ can be split into 
    $X=\begin{pmatrix} X^R & (X^{FR})^T\\ X^{FR} & X^F \end{pmatrix}$ with $X^R \in \mathbb R^{|I^2_R|\times|I^2_R|}$, $X^{FR} \in \mathbb R^{|I^2_F|\times|I^2_R|}$ and $X^F \in \mathbb R^{|I^2_F|\times|I^2_F|}$.

    With this notation, the constraints~\eqref{eq:BOR-down} and~\eqref{eq:BOR-rest} of $\myrefmath{SDP(\varepsilon,I^-,I^+)}$ are equivalent to $x^F = \sigma^F$. From Observation~\ref{obs:red} it then follows that $X^{FR} = \sigma^F(x^R)^T$ and $X^F = \sigma^F(x^F)^T =\sigma^F(\sigma^F)^T$ hold.
    As a consequence, it is easy to see that constraint~\eqref{eq:MO-SRFLP-diag} is equivalent to~\eqref{eq:SDPR-diag}. Moreover, the constraint~\eqref{eq:MO-SRFLP-3Cycle} is equivalent to~\eqref{eq:SDPR-3Cycle} by using the definition of $\widetilde{X}$. 
    Furthermore, by easy computations that exploit the properties of the Frobenius inner product~\eqref{eq:BOR-obj} and~\eqref{eq:BOR-eps} and equivalent to~\eqref{eq:SDPR-obj} and~\eqref{eq:SDPR-eps}.
    
\newcommand{\myVecE}{a} 
\newcommand{\myVecR}{b} 
\newcommand{\myVecF}{c} 
Finally, for each $\myVecE \in \mathbb{R}$, $\myVecR \in \mathbb R^{|I^2_R|}$ and $\myVecF \in \mathbb R^{|I^2_F|}$ it holds that 
\begingroup
\allowdisplaybreaks
\begin{align*}
    \begin{pmatrix} \myVecE \\ \myVecR \\ \myVecF \end{pmatrix}^T
    &\begin{pmatrix}
        1   & (x^R)^T & (x^F)^T \\
        x^R & X^R   & (X^{FR})^T \\
        x^F & X^{FR} & X^F  \end{pmatrix}
    \begin{pmatrix}\myVecE \\ \myVecR \\ \myVecF \end{pmatrix} 
    = \begin{pmatrix} \myVecE \\ \myVecR \\ \myVecF \end{pmatrix}^T 
    \begin{pmatrix*}[l]
        \myVecE + (x^R)^T \myVecR + (x^F)^T \myVecF \\ 
        \myVecE x^R + X^R \myVecR + (X^{FR})^T \myVecF\\ 
        \myVecE x^F + X^{FR} \myVecR + X^F \myVecF \end{pmatrix*}\\
    = & \myVecE^2 + \myVecE (x^R)^T \myVecR + \myVecE (x^F)^T \myVecF 
     + \myVecE \myVecR^T x^R + \myVecR^T X^R \myVecR + \myVecR^T (X^{FR})^T \myVecF 
     + \myVecE \myVecF^T x^F + \myVecF^T X^{FR} \myVecR + \myVecF^T X^F \myVecF\\
    = & \myVecE^2 + 2 \myVecE \myVecR^T x^R + 2 \myVecE \myVecF^T x^F 
      + \myVecR^T X^R \myVecR + 2 \myVecF^T X^{FR} \myVecR 
      + \myVecF^T X^F \myVecF\\ 
    = & \myVecE^2 + 2 \myVecE \myVecR^T x^R + 2 \myVecE \myVecF^T \sigma^F 
      + \myVecR^T X^R \myVecR + 2 \myVecF^T \sigma^F(x^R)^T \myVecR 
      + \myVecF^T \sigma^F(\sigma^F)^T \myVecF\\ 
    = & (\myVecE + \myVecF^T \sigma^F)^2 + 2 (\myVecE + \myVecF^T \sigma^F) \myVecR^T x^R 
      + \myVecR^T X^R \myVecR \\ 
    = & \begin{pmatrix} \myVecE + \myVecF^T \sigma^F \\ \myVecR \end{pmatrix}^T
    \begin{pmatrix}
        1 & (x^R)^T \\
        x^R & X^R \\
    \end{pmatrix} 
    \begin{pmatrix} \myVecE + \myVecF^T \sigma^F \\ \myVecR \end{pmatrix},
\end{align*}
\endgroup
implying that~\eqref{eq:MO-SRFLP-psd} is equivalent to~\eqref{eq:SDPR-psd}. Thus, $\myrefmath{SDP(\varepsilon,I^-,I^+)}$ is equivalent to $\myrefmath{SDP^R(\varepsilon,I^-,I^+)}$.
\end{proof}

As a result of Proposition~\ref{prop:red}, for each fixed variable within a \BB node we can remove the corresponding entry of $x$ and the corresponding row and columns from $X$ and only have to solve the reduced SDP $\myrefmath{SDP^R(\varepsilon,I^-,I^+)}$.
Note that $\myrefmath{SDP^R(\varepsilon,I^-,I^+)}$ is equivalent to replacing all variables by their value determined in the variable fixings and utilizing all implications of doing so in $\myrefmath{SDP(\varepsilon,I^-,I^+)}$ for both~$x$ and~$X$, which is explicitly visible in~\eqref{eq:fixingImplications}.

We note that similar approaches for reductions have been used in the code of the branch-and-bound solvers BiqCrunch by \citet{BiqCrunch} and BiqBin by \citet{BiqBin2022} for binary quadratic problems; however, their respective reduction procedures have not been described in the respective papers. For binary quadratic problems, the relaxation reduction was shown formally in \citet{Buchheim2012}. Furthermore, it is well known that the matrix formulation for the max-cut problem can be reduced in the dimension of the variable matrix, see for example \citet{Helmberg1998}. 

\section{Multi-objective enhancements}\label{sec:MO-Enhance}
In this section we present two enhancements of our solution algorithm for obtaining the nondominated set of the \BOSRFLP which are specific for the multi-objective setting, namely a tree reusing technique to warm start the \BB algorithms for different values of $\varepsilon$ and a solution pooling.

\subsection{Tree reusing}\label{sec:tree}
In our solution approach based on the $\varepsilon$-constraint method detailed in Section~\ref{sec:algorithm}, within each iteration we solve $\myrefmath{BOSRFLP(\varepsilon)}$ for a different value of $\varepsilon$. 
With our \emph{tree reusing} we exploit this 
to provide an extensive warm start for each iteration except the first one. This warm start consists of not starting the \BB algorithm from scratch to solve $\myrefmath{BOSRFLP(\varepsilon)}$ for the $\varepsilon$ of this iteration, but initializing it with \BB nodes from the previous iteration with the previous $\varepsilon$. 

Let $N(\varepsilon)$ be the set of $(I^-,I^+)$ of \BB nodes which got pruned by integrality (line~\ref{line:integrality}) or pruned by bound (lines~\ref{line:bound1},~\ref{line:bound3} and~\ref{line:bound2}) after the execution of Algorithm~\ref{alg:BnB} when solving $\myrefmath{BOSRFLP(\varepsilon)}$ for a given $\varepsilon$. Moreover, let $F(\varepsilon)=\{(x,X):\eqref{eq:BO-eps}-\eqref{eq:BO-rest} \}$ be the feasible region of $\myrefmath{BOSRFLP(\varepsilon)}$ and $F(\varepsilon,I^-,I^+)=\{(x,X):\eqref{eq:BOR-eps}-\eqref{eq:BOR-rest}\}$ be the feasible region of $\myrefmath{SDP(\varepsilon,I^-,I^+)}$ for given $I^-,I^+ \subseteq I^2$. Using this notation, we can make the following simple observation.

\begin{observation}\label{ob:obseps}
    We have that $F(\varepsilon') \subseteq \bigcup_{(I^-,I^+) \in N(\varepsilon)} F(\varepsilon,I^-,I^+) $ for all $\varepsilon'\leq \varepsilon$.
\end{observation}
\begin{proof}
Clearly, we have $F(\varepsilon) \subseteq \bigcup_{(I^-,I^+) \in N(\varepsilon)} F(\varepsilon,I^-,I^+)$ by the correctness of \BB, as the only \BB-nodes we are not considering in the union are the ones which give an infeasible relaxation $\myrefmath{SDP(\varepsilon,I^-,I^+)}$. Moreover, $F(\varepsilon')\subseteq F(\varepsilon)$ for $\varepsilon'\leq\varepsilon$ is obvious, as both feasible regions stem from the same constraints, except that~\eqref{eq:BO-eps} is more restrictive for $F(\varepsilon')$. Combining both of these facts gives the observation.
\end{proof}

From Observation~\ref{ob:obseps} we get the following theorem, where we have $F(\varepsilon',I^-,I^+)$ on the right-hand-side instead of $F(\varepsilon,I^-,I^+)$.
\begin{theorem}\label{thm:thmeps}
We have that $F(\varepsilon') \subseteq \bigcup_{(I^-,I^+) \in N(\varepsilon)} F(\varepsilon',I^-,I^+) $ for all $\varepsilon'\leq \varepsilon$.
\end{theorem}

\begin{proof}
Due to Observation~\ref{ob:obseps} it suffices to show that there is no solution $(x,X) \in F(\varepsilon')$ which is in some $F(\varepsilon,I^-,I^+)$ but not in any $F(\varepsilon',I^-,I^+)$. It is easy to see that this is not possible, as the constraints defining $F(\varepsilon,I^-,I^+)$ and $F(\varepsilon',I^-,I^+)$ are the same, with the only difference being the right-hand side of~\eqref{eq:BOR-eps}, and each solution in $F(\varepsilon')$ has to fulfill~\eqref{eq:BOR-eps} in the more restrictive version in $F(\varepsilon',I^-,I^+)$.
\end{proof}
From Theorem~\ref{thm:thmeps}, we get that we can initialize $\mathcal Q$ in line~\ref{line:init} of the Algorithm~\ref{alg:BnB} for solving $\myrefmath{BOSRFLP(\varepsilon')}$ with $N(\varepsilon)$ for any $\varepsilon>\varepsilon'$ and the algorithm remains correct. In our implementation, we initialize $\mathcal Q$ with $N(\varepsilon)$ from the previous iteration with the previous value for $\varepsilon$. As priority $z^*$ for each $(I^-,I^+) \in N(\varepsilon)$ we use the optimal objective function value of the respective relaxation from the previous iteration $\myrefmath{SDP(\varepsilon,I^-,I^+)}$. In the case of non-binary branching, we inherit the same values for $\bar{I}$ and $left$ for each node from the corresponding node of the previous iteration.

Note that by checking if the optimal objective function value $z^*$ of the respective relaxation from the previous iteration is better than the objective function value of the current incumbent $z^{UB}$ in line~\ref{line:bound1} in Algorithm~\ref{alg:BnB} before solving the SDP relaxations, we avoid unnecessary computations because the node (and thus all remaining nodes) can potentially be pruned by bound without solving the respective SDP relaxation.

We note that in \citet{KLEIN1982} a similar reusing technique was proposed within a disjunctive programming framework.
This reusing technique was applied for a \BB algorithm in \citet{SAYIN1999435}, where the authors solved two ILPs per $\varepsilon$-iteration and therefore had to manage two \BB trees, whereas we solve only one integer program in each iteration.

\subsection{Solution pooling}\label{sec:pooling}
Another technique we use which exploits the iterative nature of the $\varepsilon$-constraint method is \emph{solution pooling}. For that, we store all the feasible solutions $(x,X)$ we find when solving $\myrefmath{BOSRFLP(\varepsilon)}$ for some~$\varepsilon$, i.e., solutions obtained in the primal heuristic described in Section~\ref{sec:heur} and solutions $(x,X)$ of the SDP relaxation with $x \in \{-1,1\}^{\binom{n}{2}}$, in a set $S$. 
Before the start of a new iteration of the $\varepsilon$-constraint method, we check $S$ for a good starting solution, i.e., we determine
\[(x^*,X^*) \in \argmin_{(x,X) \in S} \langle C_1,X \rangle \text{ s.t. } K_2 + \langle C_2,X \rangle \leq \varepsilon\]
and initialize our \BB with this solution $(x^*,X^*)$ (if any). We also remove all solutions $(x,X)$ with $K_2 + \langle C_2,X \rangle > \varepsilon$ from~$S$, because they will not be feasible in all remaining iterations.

\section{Computational results}\label{sec:comp-results}

We implemented our solution algorithm in Julia using the JuMP modeling framework for mathematical programming developed by \citet{Lubin2023}. The solver MOSEK 11.1.11 was used to solve the SDP relaxations in line~\ref{line:solve} of Algorithm~\ref{alg:BnB}. All settings of MOSEK were left on their default values. The computations were performed on a single core of an Intel Xeon~X5770 CPU with 2.93~GHz and 6~GB of RAM, and the time limit for each run was set to 18000~seconds (5~hours).

We utilize the just-in-time compiler of Julia to cache machine code and reduce subsequent execution times. To ensure that our computational results are repeatable, before solving each instance we execute our algorithm with a small ``warm-up'' instance and we do not include the time of solving this ``warm-up'' instance into our results. 

In all instances that we tested, the costs $c_{ij}^q$ are non-negative integers and the lengths $\ell_i$ are positive integers, thus the smallest possible objective function value difference mentioned in~\eqref{eq:delta} is at least one, so we use $\delta=1$ in all computations.

\subsection{Instances}\label{sec:instances}
As our work is the first to consider the \BOSRFLP, there are no benchmark instances from the literature. We thus created instances by i) combining existing instances for the single-objective \SRFLP and by ii) following and adapting instance creation strategies for the \SRFLP in the following way:
\begin{enumerate}[label=(\roman*)]
\item \texttt{combined} instances: We selected an existing single-objective instance (see Table~\ref{tab:litInstances} for an overview, and note that the number in the instance name represents $n$) as the primary instance and used the facility lengths as well as the pairwise cost between the facilities from this instance as $\ell_i$ and~$c^1_{ij}$. We then paired this primary instance with another existing single-objective instance of the same size, from which we used the pairwise cost between the facilities as $c^2_{ij}$. For example, the \SRFLP instances S9 and S9H from \citet{SIMMONS1969} were merged to form the two \BOSRFLP instances S9\_S9H and S9H\_S9, where the first instance name indicates the primary instance.
Note that these two \BOSRFLP instances do not coincide, because the lengths stem from two different \SRFLP instances and thus their nondominated sets differ in general. We only considered instances with fewer than 20 facilities, as preliminary computations showed that a larger number of facilities led to instances that could not be solved within our considered time limit. We obtained 13 instances using this approach.\footnote{Note that the instances AM15 and H15 consist of the same values for $c_{ij}$ for $i,j \in I, i<j$ and only differ in the lengths of the facilities $\ell_i$ for $i\in I$. Therefore, combining these two instances would result in single-objective instances again. Also we only need to combine SRFLP15 with either AM15 or H15.}
\begin{table}[htb]
    \centering
    \begin{tabular}{llr}
    \toprule
        instance & source & density\\
        \midrule
        S9 &  \citet{SIMMONS1969} & 0.8611\\
        S9H &  \citet{SIMMONS1969} & 1.0000\\
        SRFLP9 & \citet{Hungerlander2012Tech} & 0.5556\\
        S10 & \citet{SIMMONS1969} & 0.8000\\
        SRFLP10 & \citet{Hungerlander2012Tech} & 0.6222\\
        LW11 & \citet{LOVE1976}& 0.8727\\
        S11 & \citet{SIMMONS1969} & 0.8727\\
        AM15 & \citet{AMARAL2006} & 0.7143\\
        H15 & \citet{HERAGU1991} & 0.7143\\
        SRFLP15 & \citet{Hungerlander2012Tech} & 0.5238\\      
        \bottomrule
    \end{tabular}
    \caption{Instances from the literature}
    \label{tab:litInstances}
\end{table}
\item \texttt{random} instances: We followed the approach of \citet{PALUBECKIS2017} for the \SRFLP and created instances for each $n \in \{10, 11, \dots, 20\}$. In this approach, the lengths $\ell_i$ for $i \in I$ are randomly drawn as uniformly distributed integers from $1$ to $r$, where $r \in \{10,20\}$. The pairwise costs of the facilities $c_{ij}^q$ for $i,j \in I, i < j$ and $q \in \{1,2,\dots,p\}$ are generated by selecting $\binom{n}{2}$ random integers in the range from $0$ to $r$, where $0$ is selected with probability $1-d/100$, and each positive integer from $1$ to $r$ with probability $d/(100r)$, where the density $d\in\{50,70,90\}$, as usual in literature instances, see Table~\ref{tab:litInstances}. An instance is denoted by R\_$n$\_$d$\_$r$\_$id$ where $id \in \{1, 2\}$ is a number to distinguish instances which were created using the same parameter values. We obtained 132 instances this way.
\end{enumerate}

\subsection{Investigation of the impact of our enhancements}\label{sec:ingredients}
In this section, we analyze the impact of our algorithm's enhancements on its performance. 
As basic setting we selected the \texttt{mostInfeasible} branching strategy, which is a standard branching strategy for \BB algorithms. Later, a comparison of the computational performance of the different branching strategies is provided in Section~\ref{sec:branchAgg}. For now, we consider the following five settings including more and more enhancements for our algorithm:

\begin{itemize}[noitemsep]
    \item \texttt{I}: Algorithm~\ref{alg:epsMethod} incorporating Algorithm~\ref{alg:BnB} using the \texttt{mostInfeasible} branching strategy without any enhancements
    \item \texttt{IS}: Setting \texttt{I} with the sequential fixing as described in Section~\ref{sec:sequential}
    \item \texttt{ISR}: Setting \texttt{IS} with the reduction as described in Section~\ref{sec:reduction}
    \item \texttt{ISRT}: Setting \texttt{ISRT} with the tree reusing as described in Section~\ref{sec:tree}
    \item \texttt{ISRTP}: Setting \texttt{ISRTP} with the solution pooling as described in Section~\ref{sec:pooling}
\end{itemize}

\subsubsection{Aggregated results}\label{sec:ing-agg}
First, we present aggregated performance results of the different settings of our algorithm across all instances. Detailed computational results of the settings per instance are provided in Appendix~\ref{app:ingredients}.

\begin{figure}[htb]
    \centering
    \includegraphics[width=0.6\linewidth]{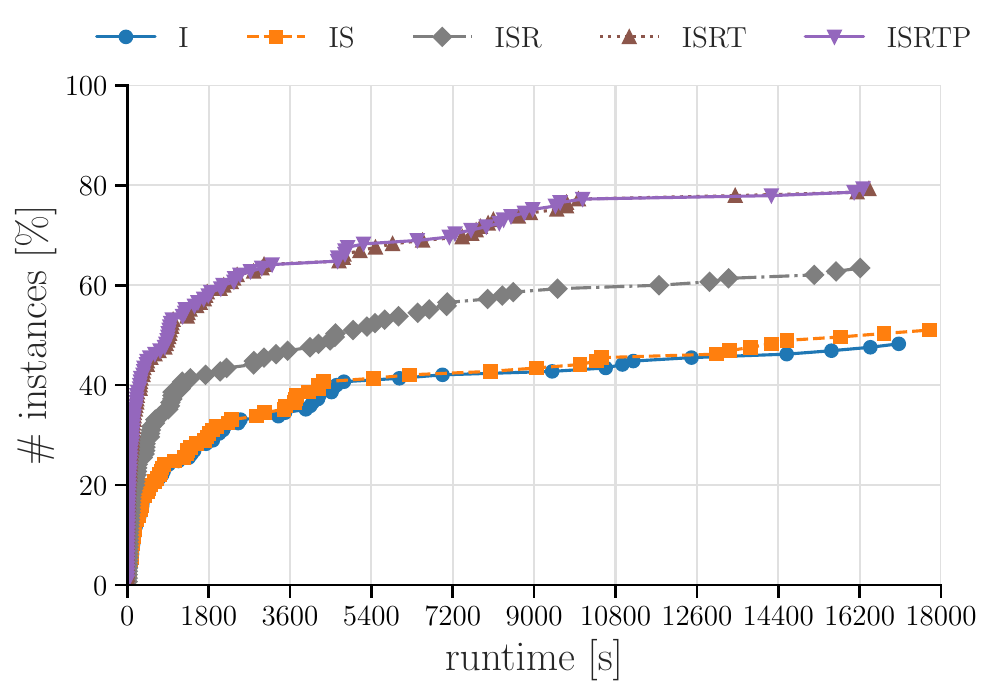}
    \caption{ECDFs of the runtimes for different settings of our algorithm over \texttt{combined} and \texttt{random} instances}
    \label{fig:features}
\end{figure}

In Figure~\ref{fig:features}, we show a plot of empirical cumulative distribution functions (ECDFs) with respect to the runtimes for the five settings over the \texttt{combined} and \texttt{random} instances accumulated. The figure shows that clearly the basic setting \texttt{I} performs worst, as we obtained the nondominated sets for only about 48\% of the instances. With the sequential fixing (setting \texttt{IS}), we observed only a minor improvement, obtaining the nondominated sets for about 51\% of the instances. However, combined with the reduction (setting \texttt{ISR}), for more than 63\% of the instances the nondominated sets could be obtained within the time limit. With our main contribution, the tree reusing (setting \texttt{ISRT}), the performance improved again significantly, and for nearly 80\% of all instances the nondominated sets were obtained. The solution pooling (setting \texttt{ISRTP}) reduced the runtime of the algorithm slightly.

\subsubsection{Results for one sample instance}\label{sec:ing-sample}

Next, we consider one exemplary instance, namely the instance R\_13\_70\_10\_2, in more detail to further analyze the impact of the enhancements.

\begin{figure}[h!tb]
    \begin{subfigure}[t]{0.48\textwidth}
        \centering
        \includegraphics[width=\textwidth]{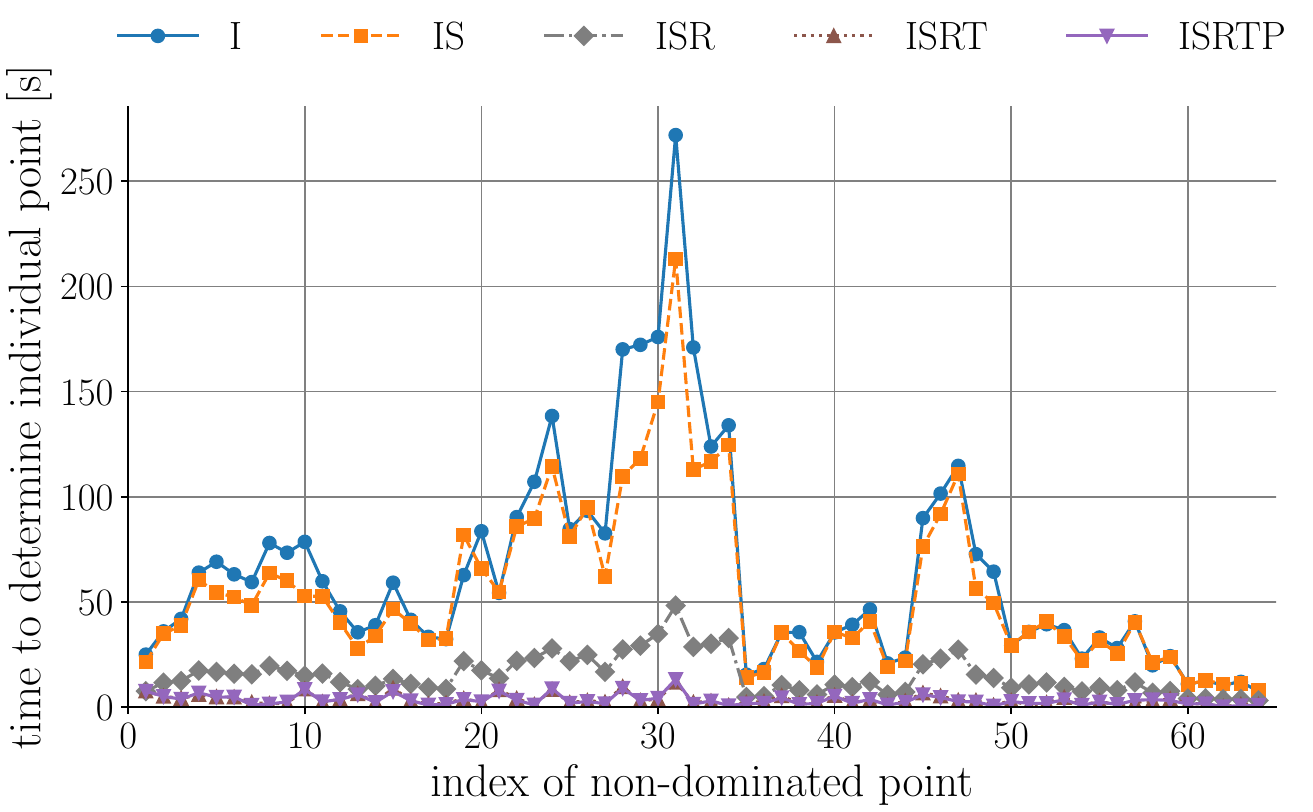}
        \caption{Time}
        \label{fig:features_instance_time}
    \end{subfigure}
    \begin{subfigure}[t]{0.48\textwidth}
        \centering
        \includegraphics[width=\textwidth]{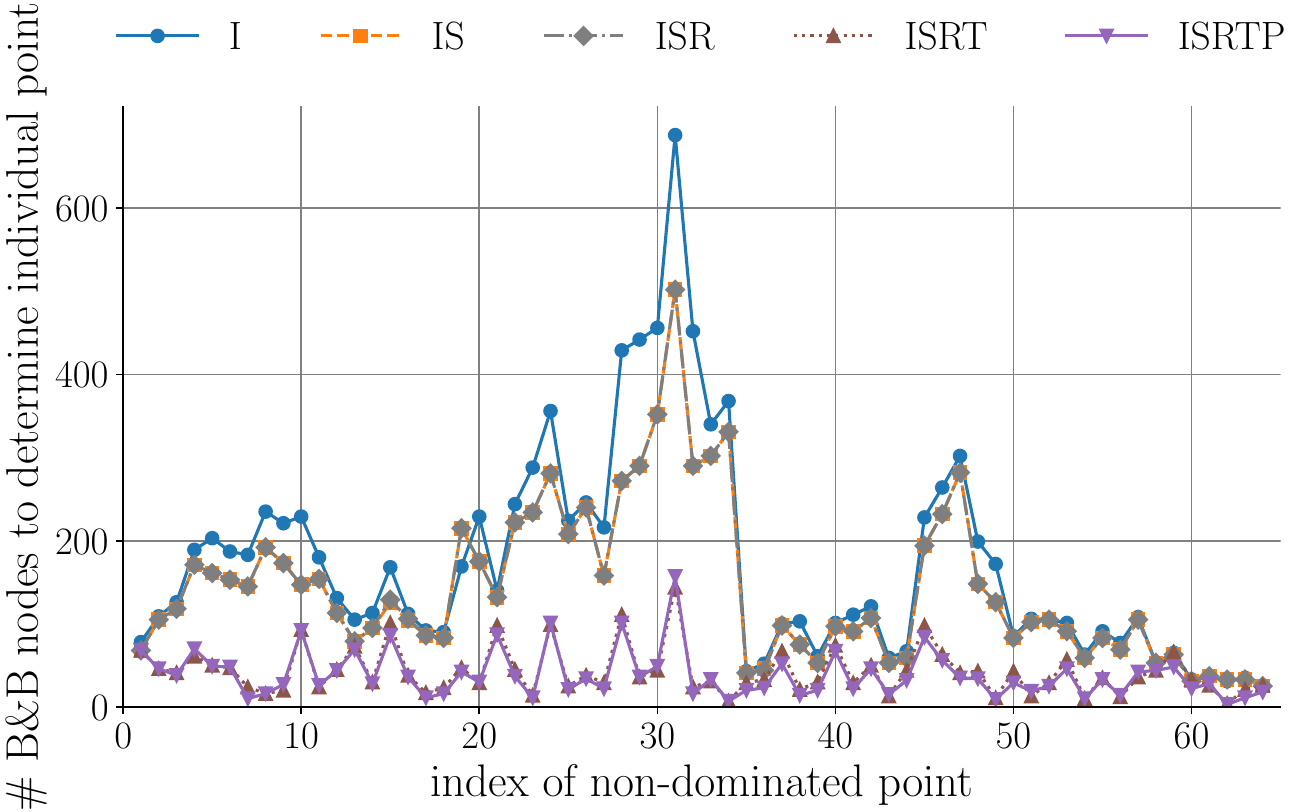}
        \caption{Number of \BB nodes}
        \label{fig:features_instance_nodes}
    \end{subfigure}
    \caption{Time (a) and number of \BB nodes (b) processed to determine the individual nondominated points for different settings of our algorithm for the instance R\_13\_70\_10\_2}
\end{figure}

Figure~\ref{fig:features_instance_time} shows the time it takes different settings of our algorithm to determine the individual nondominated points, i.e., the $x$-axes represents the indices of the order in which the nondominated points were found and the $y$-axes represents the time from finding the previous nondominated point to finding this nondominated point.
In this figure, we see that with the setting \texttt{IS} the times per nondominated point were reduced by up to a third compared to the basic setting \texttt{I}. Note that the sequential fixing can lead to different branching decisions and in rare cases increase the times. With the setting \texttt{ISR}, the times per nondominated point are significantly reduced, from up to 213 seconds per nondominated point using \texttt{IS}, to at most 50 seconds. Additionally, with the setting \texttt{ISRT}, the times per nondominated point are further reduced substantially, down to at most 12 seconds. With the setting \texttt{ISRTP}, the times could not be further reduced significantly, indicating that the primal heuristic finds a good feasible solution and thus also provides a good upper bound early on. 

In Figure~\ref{fig:features_instance_nodes}, we plot the number of \BB nodes processed by different settings of our algorithm to determine individual nondominated points for the instance R\_13\_70\_10\_2, i.e., the $x$-axes represents the indices of the order in which the nondominated points were found and the $y$-axes represents the number of \BB nodes that were processed between finding the previous nondominated point to finding this nondominated point.
This figure highlights that the enhancements of our algorithm influence its runtime in different ways. In particular, with the setting \texttt{IS} less \BB nodes were processed per nondominated point than with the basic setting \texttt{I}. However, using reduction in setting \texttt{ISR} did not decrease the number of \BB nodes processed per nondominated point further, indicating that the huge time differences between \texttt{IS} and \texttt{ISR} are due to the smaller SDP-relaxations being solved much faster within \texttt{ISR}. With the setting \texttt{ISRT}, the number of \BB nodes processed per nondominated point was further reduced significantly, from up to 500 \BB nodes per point to less than 100 for most points. With the setting \texttt{ISRTP} the number of \BB nodes processed per nondominated point could not be further decreased significantly.

\subsection{Investigation of the impact of the branching strategy}\label{sec:branchResult}

Next, we compare all our branching strategies and then examine the influence of all our enhancements on the best branching strategy determined.

\subsubsection{Aggregated results for all branching strategies}\label{sec:branchAgg}

The previous figures and tables have shown that each enhancement incrementally improved the runtime of our algorithm, thus we use all enhancements for the comparison of our branching strategies. In particular, besides the setting \texttt{ISRTP}, we consider the following settings:
\begin{itemize}[noitemsep]
    \item \texttt{LSRTP}: Setting \texttt{ISRTP} with the \texttt{lengthWeighted} branching strategy 
    \item \texttt{CSRTP}: Setting \texttt{ISRTP} with the \texttt{connected} branching strategy
    \item \texttt{FSRTP}: Setting \texttt{ISRTP} with the \texttt{facilityFocused} branching strategy
    \item \texttt{DSRTP}: Setting \texttt{ISRTP} with the \texttt{disconnected} branching strategy
    \item \texttt{NBRTP}: Setting \texttt{ISRTP} with the \texttt{nonBinary} branching strategy without the (as detailed in Section~\ref{sec:sequential}) redundant sequential fixing
\end{itemize}

\begin{figure}[h!tb]
    \centering
    \includegraphics[width=0.6\linewidth]{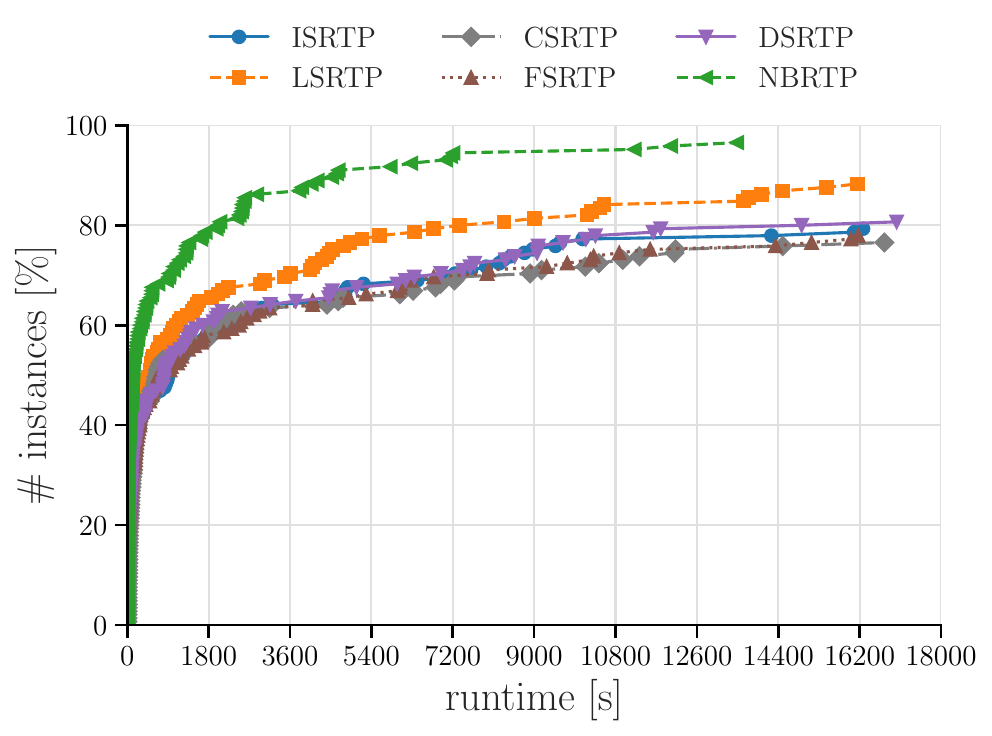}
    \caption{ECDFs of the runtimes for different branching strategies of our algorithm over \texttt{combined} and \texttt{random} instances}
    \label{fig:branching}
\end{figure}

In Figure~\ref{fig:branching}, we compare these six settings for our algorithm by showing ECDFs of the runtimes. The settings \texttt{CSRTP}, \texttt{FSRTP} and \texttt{DSRTP} performed similarly to \texttt{ISRTP}. The setting \texttt{LSRTP} using the \texttt{lengthWeighted} branching strategy was the best performing binary branching strategy, obtaining the nondominated sets for 89\% of the instances. However, the non-binary branching strategy was the overall best branching strategy and we obtained the nondominated sets for about 96\% of all instances with the setting \texttt{NBRTP}. Thus, our computational study showed that \texttt{lengthWeighted} is the best binary branching strategy, and \texttt{nonBinary} is the overall best branching strategy.

\subsubsection{Results for enhancements for the best branching strategy for one sample instance}\label{sec:branchSample}
In this section, we analyze the impact of the enhancements of our algorithm on our best branching strategy \texttt{nonBinary} in detail. To do so, we compared the following settings of our algorithm for the instance R\_13\_70\_10\_2: 
\begin{itemize}[noitemsep]
    \item \texttt{NB}: Algorithm~\ref{alg:epsMethod} incorporating Algorithm~\ref{alg:BnB} using the \texttt{nonBinary} branching strategy without any enhancements
    \item \texttt{NBR}: Setting \texttt{NB} with the reduction as described in Section~\ref{sec:reduction}
    \item \texttt{NBRT}: Setting \texttt{NBR} with the tree reusing as described in Section~\ref{sec:tree}
    \item \texttt{NBRTP}: Setting \texttt{NBRT} with the solution pooling as described in Section~\ref{sec:pooling}
\end{itemize}

\begin{figure}[h!tb]
    \centering
    \begin{subfigure}[t]{0.48\textwidth}
        \centering
        \includegraphics[width=\textwidth]{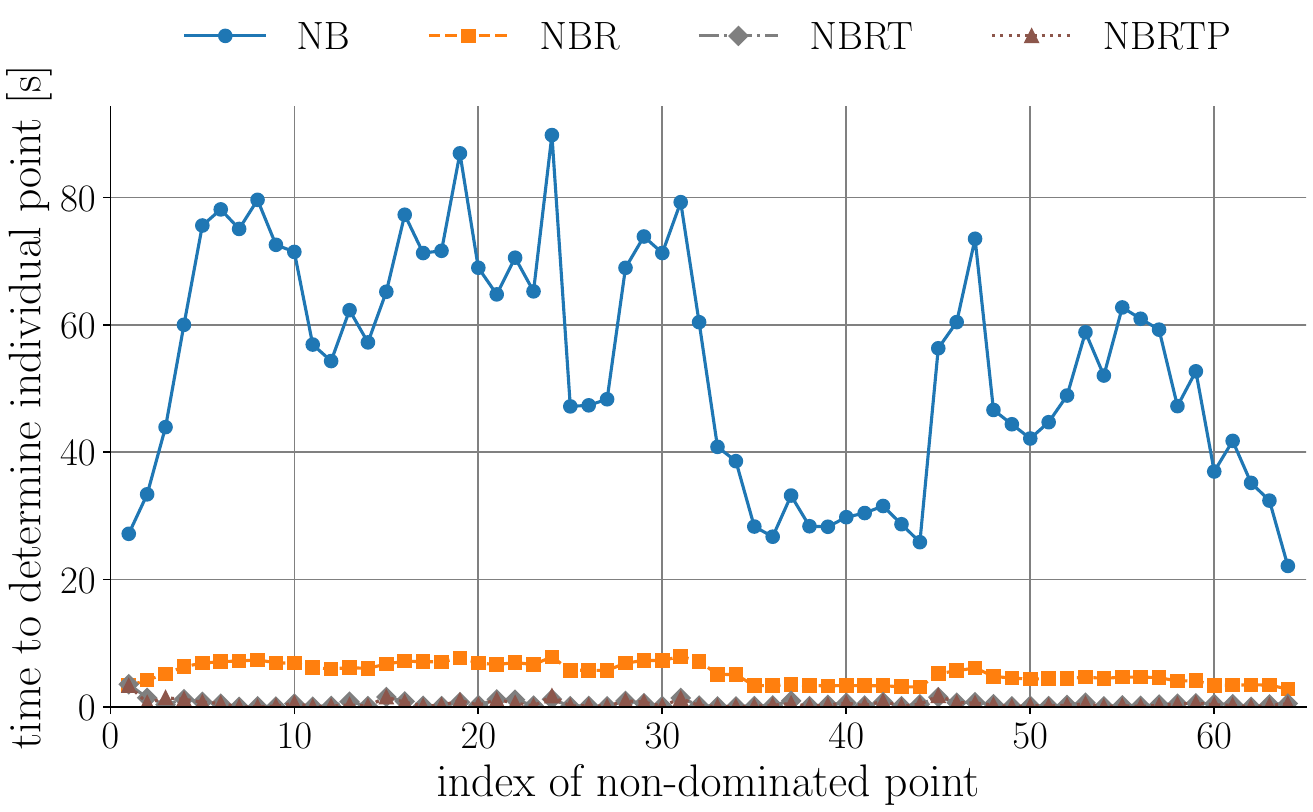}
        \caption{Time }
        \label{fig:features_instance_time_NB}
    \end{subfigure}
    \hfill 
    \begin{subfigure}[t]{0.48\textwidth}
        \centering
        \includegraphics[width=\textwidth]{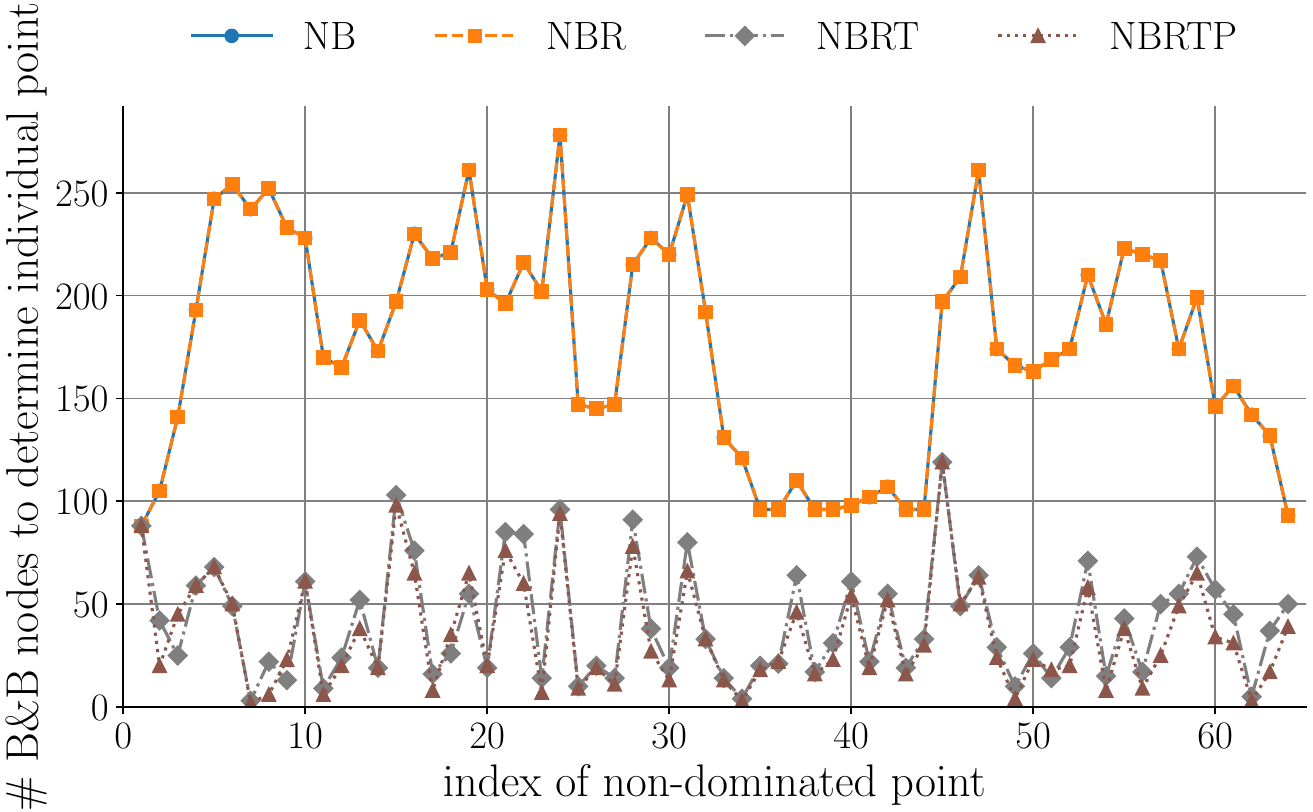}
        \caption{
        Number of \BB nodes}
        \label{fig:features_instance_nodes_NB}
    \end{subfigure}
    \caption{Time (a) and number of \BB nodes processed (b) to determine the individual nondominated points for different settings of our algorithm with \texttt{nonBinary} branching for the instance R\_13\_70\_10\_2}
\end{figure}

In Figure~\ref{fig:features_instance_time_NB}, created analogously to Figure~\ref{fig:features_instance_time}, we see that the times with the reduction (setting \texttt{NBR}), are shorter then the ones for \texttt{NB} by up to 90\% and most nondominated points were determined within less than 10 seconds. The setting \texttt{NBRT} with tree reusing determined most nondominated points within less than 1 second. The setting \texttt{NBRTP} did not improve the times significantly.

Figure~\ref{fig:features_instance_nodes_NB} was produced analogously to Figure~\ref{fig:features_instance_nodes} and shows that with the reduction (setting \texttt{NBR}), the same number of \BB nodes were processed per nondominated point as with the setting \texttt{NB}. Thus, both for binary and non-binary branching we observe that using the reduction makes solving the SDPs significantly faster. With the tree reusing (setting \texttt{NBRT}), the number of \BB nodes processed per nondominated point was significantly lower than with \texttt{NBR}. Finally, for some of the nondominated points, less \BB nodes were processed with solution pooling (setting \texttt{NBRTP}) than with \texttt{NBRT}.

\subsection{Comparison with off-the-shelf ILP solvers}\label{sec:solver-comp}
In this section, we compare the best setting of our algorithm with off-the-shelf ILP solvers. Towards this end, we use Algorithm~\ref{alg:epsMethod}, and within \texttt{solve$(P(\varepsilon))$} we employ the off-the-shelf ILP solvers CPLEX~22.1 and Gurobi~12 to solve~\eqref{eq:BOSRFLPe-IP} (opposed to using our \BB algorithm detailed in Algorithm~\ref{alg:BnB} to solve~$\myrefmath{P(\varepsilon)}$ as in our algorithm). We set the parameters \texttt{MIPGap} and \texttt{AbsMIPGap} for CPLEX and \texttt{MIPGap} and \texttt{MIPGapAbs} for Gurobi to zero and to $1-10^{-6}$, respectively. These parameter settings ensure that the ILP solvers only prune \BB nodes by bound if the absolute gap between lower and upper bound is smaller than~$1$. All other parameters were left at their default values. 

\begin{figure}[h!tb]
    \centering
    \includegraphics[width=0.6\linewidth]{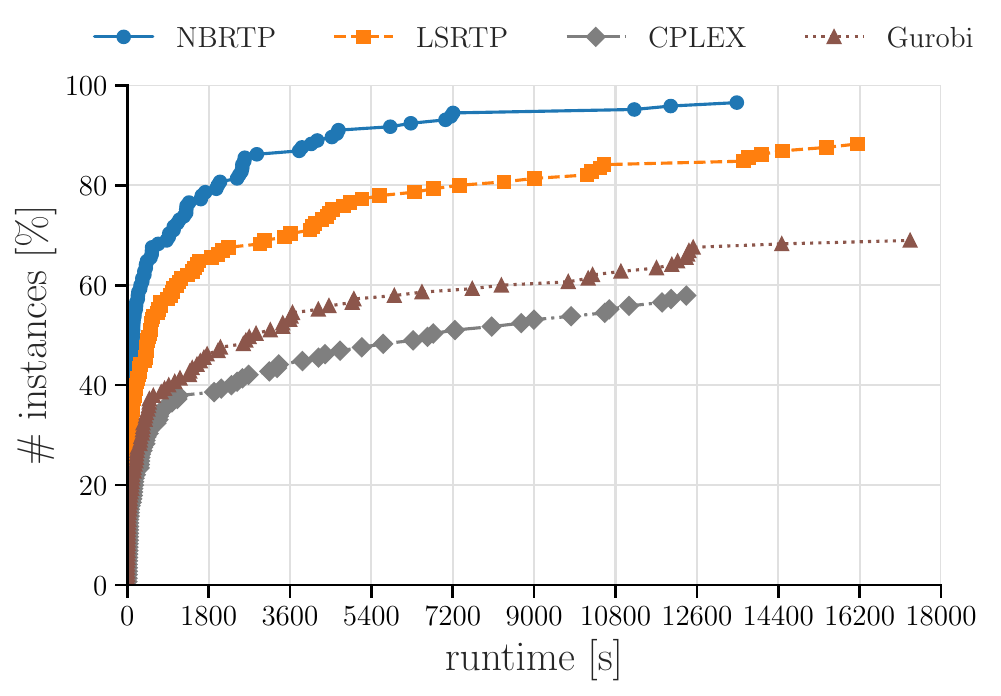}
    \caption{ECDFs of the runtimes of different settings of our algorithm and off-the-shelf ILP solvers over \texttt{combined} and \texttt{random} instances}
    \label{fig:IP_comp}
\end{figure}

In Figure~\ref{fig:IP_comp}, we compare our best binary branching setting \texttt{LSRTP} and our overall best setting, the non-binary branching setting \texttt{NBRTP} with the ILP-based approaches by showing ECDFs of the runtimes. Our algorithm outperformed both ILP-based approaches with both the binary and the non-binary branching strategy using all enhancements. In particular, the best ILP-based approach obtained the nondominated sets for only 69\% of the instances within the time limit, while our algorithm managed to obtain the nondominated sets for 96\% of all instances.

In Tables~\ref{tab:off_the_shelf_lit} and~\ref{tab:off_the_shelf_rand}, we report the runtime $t$ in seconds and the number of nondominated points~$|\mathcal{Y}_N|$ found at termination of the algorithm for the \texttt{combined} and \texttt{random} instances, respectively, for the settings \texttt{I}, \texttt{LSRTP} and \texttt{NBRTP} of our algorithm, as well as for the ILP-based approaches using the solvers CPLEX and Gurobi. For each instance, the fastest runtime is marked bold, and if no approach finished within the time limit, the highest number of nondominated points found is marked bold. Runs that terminated because of the time limit are marked with TL. 

The setting \texttt{NBRTP} outperformed both ILP-based approaches in every instance in terms of runtime. For the instance R\_12\_70\_10\_2 our algorithm with the setting \texttt{NBRTP} was faster than the basic setting~\texttt{I} by a factor of 180 and for the instance R\_16\_90\_20\_1 the setting \texttt{NBRTP} was faster than Gurobi by a factor of 90. The basic setting~\texttt{I} did not manage to obtain a single nondominated point for some instances, e.g., R\_19\_70\_10\_1, while our best setting obtained the nondominated set within less than an hour for this instance. Gurobi turned out to be the best ILP-based approach, obtaining at least one nondominated point for all instances. However, for the instance R\_20\_90\_20\_1, it managed to obtain only 3\% of the nondominated points, while our algorithm with the setting \texttt{NBRTP} obtained the nondominated set within less than an hour for this instance. As a result, the best settings with both binary and non-binary branching of our algorithm significantly outperformed all approaches based on off-the-shelf ILP solvers.

\input{tables/ILPComparison/ILP_comp_lit_table}
\input{tables/ILPComparison/ILP_comp_rand_table}

\section{Conclusions and outlook}\label{sec:conclusion}

In this work, we introduce a multi-objective version of the well-known single-row facility layout problem (\SRFLP) using Pareto optimality as a concept of optimality. 
We design a solution algorithm to solve the bi-objective \SRFLP based on the $\varepsilon$-constraint method. Our solution algorithm uses our own SDP-based branch-and-bound procedure for solving the individual iterations of the $\varepsilon$-constraint method, as SDP relaxations are known to be more effective for solving the \SRFLP compared to linear relaxations. This is in contrast to many existing works on the $\varepsilon$-constraint method, which use ILP solvers in a black-box fashion for solving the problems arising at individual iterations. This approach enables us to propose several enhancements for our $\varepsilon$-constraint approach, such as non-binary branching and reusing of nodes in the branch-and-bound trees, which are usually not possible to consider when using black-box solvers. 

The effectiveness of our solution algorithm is assessed in an extensive computational study with newly generated instances based on existing single-objective SRFLP instances and on randomly generated instances. 
This computational study shows that all our proposed enhancements incrementally improve the runtime of our solution algorithm. In particular, the tree reusing, where we initialize the \BB with a carefully selected subset of nodes of the \BB tree of the previous iteration, reduced the number of SDP relaxations required and hence the runtime significantly. Furthermore, the computational study showed that the non-binary branching strategy performed best in our solution algorithm. Our algorithm obtained the nondominated sets for 140 out of the 145 instances, while the best ILP-based $\varepsilon$-constraint method implementation using Gurobi to solve the single-objective ILP problems managed to do that only for 100 instances.

Several directions remain open for further work. One possible extension is to add valid inequalities to strengthen the SDP relaxations, as has been done for the single-objective \SRFLP in \citet{hungerlander_semidefinite_2013} and \citet{Schwiddessen2020}. In particular, the triangle inequalities are promising candidates. However, their exponential number requires the use or development of alternative SDP solvers, as the interior point method cannot handle large sets of constraints well. 
A further avenue is the generalization of the solution algorithm to more than two objectives, and in particular, to extend the tree reusing technique to such a setting. Moreover, the application of the tree reusing technique to other bi-objective problems could also be promising.

\ifArXiV
\bibliographystyle{plainnat}
\fi
\bibliography{SRFLP.bib}

\clearpage
\appendix

\section{Detailed results for the impact of our enhancements} \label{app:ingredients}
Tables~\ref{tab:features_lit} and~\ref{tab:features_rand} report the runtime $t$ in seconds and the number of nondominated points $|\mathcal{Y}_N|$ found at termination of the algorithm for the five settings introduced in Section~\ref{sec:ingredients} for the \texttt{combined} and \texttt{random} instances, respectively. The fastest runtime is marked in bold; if no approach finished within the time limit, the highest number of nondominated points found is marked in bold. 

We observe that the setting \texttt{I} performs worst and obtains only few nondominated points for instances with more than 16 facilities. The setting \texttt{IS} using the sequential fixing performs only slightly better than \texttt{I}. The setting \texttt{ISR} additionally using the reduction significantly reduced the runtimes. With the tree reusing in setting \texttt{ISRT} the runtime was again reduced significantly. With the setting \texttt{ISRTP}, using the solution pooling, the runtimes were reduced slightly for some instances and slightly worse for others. These mixed results for the solution pooling are caused by run-to-run variations, however, it shows the general trend that the primal heuristic as described in Section~\ref{sec:heur} is finding good feasible solutions early on and the solution pooling rarely provides better solutions.

\input{tables/features/features_lit_table}

\input{tables/features/features_rand_table}

\end{document}

%% file: tables/ILPComparison/ILP_comp_lit_table.tex
\begingroup\fontsize{9}{10}\selectfont
\begin{longtable}{lrrrrrrrrrr}
\caption{\label{tab:off_the_shelf_lit}Runtimes and numbers of nondominated points found for different settings of our algorithm and for off-the-shelf ILP solvers for \texttt{combined} instances}\\
\toprule
  & \multicolumn{2}{c}{\texttt{I}} & \multicolumn{2}{c}{\texttt{LSRTP}} & \multicolumn{2}{c}{\texttt{NBRTP}}  & \multicolumn{2}{c}{CPLEX} & \multicolumn{2}{c}{Gurobi} \\
  \cmidrule(lr){2-3}\cmidrule(lr){4-5}\cmidrule(lr){6-7}\cmidrule(lr){8-9}\cmidrule(lr){10-11}
\multicolumn{1}{c}{instance} &\multicolumn{1}{c}{$t \,[s]$} & \multicolumn{1}{c}{$|\mathcal{Y}_N|$} &\multicolumn{1}{c}{$t \,[s]$} & \multicolumn{1}{c}{$|\mathcal{Y}_N|$} & \multicolumn{1}{c}{$t \,[s]$} & \multicolumn{1}{c}{$|\mathcal{Y}_N|$} & \multicolumn{1}{c}{$t \,[s]$} & \multicolumn{1}{c}{$|\mathcal{Y}_N|$} & \multicolumn{1}{c}{$t \,[s]$} & \multicolumn{1}{c}{$|\mathcal{Y}_N|$} \\
\midrule
\endfirsthead
\caption[]{Runtimes and numbers of nondominated points found for different settings of our algorithm and for off-the-shelf ILP solvers for \texttt{combined} instances (continued)}\\
\toprule
  & \multicolumn{2}{c}{\texttt{I}} & \multicolumn{2}{c}{\texttt{LSRTP}} & \multicolumn{2}{c}{\texttt{NBRTP}}  & \multicolumn{2}{c}{CPLEX} & \multicolumn{2}{c}{Gurobi} \\
  \cmidrule(lr){2-3}\cmidrule(lr){4-5}\cmidrule(lr){6-7}\cmidrule(lr){8-9}\cmidrule(lr){10-11}
\multicolumn{1}{c}{instance} &\multicolumn{1}{c}{$t \,[s]$} & \multicolumn{1}{c}{$|\mathcal{Y}_N|$} &\multicolumn{1}{c}{$t \,[s]$} & \multicolumn{1}{c}{$|\mathcal{Y}_N|$} & \multicolumn{1}{c}{$t \,[s]$} & \multicolumn{1}{c}{$|\mathcal{Y}_N|$} & \multicolumn{1}{c}{$t \,[s]$} & \multicolumn{1}{c}{$|\mathcal{Y}_N|$} & \multicolumn{1}{c}{$t \,[s]$} & \multicolumn{1}{c}{$|\mathcal{Y}_N|$} \\
\midrule
\endhead

\midrule
\multicolumn{11}{r}{Continued on next page} \\
\midrule
\endfoot

\bottomrule
\endlastfoot
\input{tables/data/ILP_comp_lit}%
\end{longtable}
\endgroup{}

%% file: tables/data/ILP_comp_lit.tex
S9H\_SRFLP9 & 115 & 35 & 10 & 35 & \textbf{5} & 35 & 34 & 35 & 35 & 35 \\
S9\_S9H & 35 & 20 & 3 & 20 & \textbf{2} & 20 & 12 & 20 & 14 & 20 \\
S9H\_S9 & 27 & 16 & 4 & 16 & \textbf{2} & 16 & 8 & 16 & 6 & 16 \\
S9\_SRFLP9 & 64 & 21 & 8 & 21 & \textbf{4} & 21 & 17 & 21 & 18 & 21 \\
SRFLP9\_S9 & 37 & 10 & 4 & 10 & \textbf{2} & 10 & 5 & 10 & 9 & 10 \\
SRFLP9\_S9H & 49 & 18 & 4 & 18 & \textbf{2} & 18 & 10 & 18 & 11 & 18 \\
SRFLP10\_S10 & 106 & 21 & 8 & 21 & \textbf{5} & 21 & 25 & 21 & 23 & 21 \\
S10\_SRFLP10 & 147 & 18 & 10 & 18 & \textbf{7} & 18 & 38 & 18 & 21 & 18 \\
S11\_LW11 & 650 & 46 & 50 & 46 & \textbf{13} & 46 & 285 & 46 & 192 & 46 \\
LW11\_S11 & 1419 & 79 & 72 & 79 & \textbf{17} & 79 & 337 & 79 & 289 & 79 \\
H15\_SRFLP15 & 12481 & 62 & 511 & 62 & \textbf{111} & 62 & 3313 & 62 & 1767 & 62 \\
AM15\_SRFLP15 & 17067 & 89 & 1148 & 89 & \textbf{335} & 89 & 11099 & 89 & 4462 & 89 \\
SRFLP15\_AM15 & TL & 82 & 733 & 110 & \textbf{235} & 110 & 10663 & 110 & 5013 & 110 

%% file: tables/ILPComparison/ILP_comp_rand_table.tex
\begingroup\fontsize{9}{10}\selectfont
\begin{longtable}{lrrrrrrrrrr}
\caption{\label{tab:off_the_shelf_rand}Runtimes and numbers of nondominated points found for different settings of our algorithm and for off-the-shelf ILP solvers for \texttt{random} instances}\\
\toprule
  & \multicolumn{2}{c}{\texttt{I}} & \multicolumn{2}{c}{\texttt{LSRTP}} & \multicolumn{2}{c}{\texttt{NBRTP}}  & \multicolumn{2}{c}{CPLEX} & \multicolumn{2}{c}{Gurobi} \\
  \cmidrule(lr){2-3}\cmidrule(lr){4-5}\cmidrule(lr){6-7}\cmidrule(lr){8-9}\cmidrule(lr){10-11}
\multicolumn{1}{c}{instance} &\multicolumn{1}{c}{$t \,[s]$} & \multicolumn{1}{c}{$|\mathcal{Y}_N|$} &\multicolumn{1}{c}{$t \,[s]$} & \multicolumn{1}{c}{$|\mathcal{Y}_N|$} & \multicolumn{1}{c}{$t \,[s]$} & \multicolumn{1}{c}{$|\mathcal{Y}_N|$} & \multicolumn{1}{c}{$t \,[s]$} & \multicolumn{1}{c}{$|\mathcal{Y}_N|$} & \multicolumn{1}{c}{$t \,[s]$} & \multicolumn{1}{c}{$|\mathcal{Y}_N|$} \\

\midrule
\endfirsthead
\caption[]{Runtimes and numbers of nondominated points found for different settings of our algorithm and for off-the-shelf ILP solvers for \texttt{random} instances (continued)}\\
\toprule
  & \multicolumn{2}{c}{\texttt{I}} & \multicolumn{2}{c}{\texttt{LSRTP}} & \multicolumn{2}{c}{\texttt{NBRTP}}  & \multicolumn{2}{c}{CPLEX} & \multicolumn{2}{c}{Gurobi} \\
  \cmidrule(lr){2-3}\cmidrule(lr){4-5}\cmidrule(lr){6-7}\cmidrule(lr){8-9}\cmidrule(lr){10-11}
\multicolumn{1}{c}{instance} &\multicolumn{1}{c}{$t \,[s]$} & \multicolumn{1}{c}{$|\mathcal{Y}_N|$} &\multicolumn{1}{c}{$t \,[s]$} & \multicolumn{1}{c}{$|\mathcal{Y}_N|$} & \multicolumn{1}{c}{$t \,[s]$} & \multicolumn{1}{c}{$|\mathcal{Y}_N|$} & \multicolumn{1}{c}{$t \,[s]$} & \multicolumn{1}{c}{$|\mathcal{Y}_N|$} & \multicolumn{1}{c}{$t \,[s]$} & \multicolumn{1}{c}{$|\mathcal{Y}_N|$} \\
\midrule
\endhead
\midrule
\multicolumn{11}{r}{Continued on next page} \\
\midrule
\endfoot
\bottomrule
\endlastfoot
\input{tables/data/ILP_comp_rand}%
\end{longtable}
\endgroup{}

%% file: tables/data/ILP_comp_rand.tex
R\_10\_50\_10\_1 & 98 & 23 & 10 & 23 & \textbf{5} & 23 & 20 & 23 & 14 & 23 \\
R\_10\_50\_10\_2 & 337 & 41 & 18 & 41 & \textbf{11} & 41 & 53 & 41 & 39 & 41 \\
R\_10\_50\_20\_1 & 136 & 22 & 10 & 22 & \textbf{7} & 22 & 27 & 22 & 19 & 22 \\
R\_10\_50\_20\_2 & 187 & 27 & 14 & 27 & \textbf{7} & 27 & 40 & 27 & 26 & 27 \\
R\_10\_70\_10\_1 & 67 & 18 & 6 & 18 & \textbf{3} & 18 & 11 & 18 & 8 & 18 \\
R\_10\_70\_10\_2 & 73 & 17 & 7 & 17 & \textbf{5} & 17 & 17 & 17 & 11 & 17 \\
R\_10\_70\_20\_1 & 278 & 26 & 20 & 26 & \textbf{9} & 26 & 58 & 26 & 42 & 26 \\
R\_10\_70\_20\_2 & 316 & 33 & 13 & 33 & \textbf{5} & 33 & 37 & 33 & 29 & 33 \\
R\_10\_90\_10\_1 & 132 & 24 & 9 & 24 & \textbf{3} & 24 & 28 & 24 & 27 & 24 \\
R\_10\_90\_10\_2 & 345 & 38 & 28 & 38 & \textbf{11} & 38 & 99 & 38 & 68 & 38 \\
R\_10\_90\_20\_1 & 106 & 21 & 9 & 21 & \textbf{3} & 21 & 12 & 21 & 7 & 21 \\
R\_10\_90\_20\_2 & 153 & 13 & 14 & 13 & \textbf{6} & 13 & 35 & 13 & 22 & 13 \\
R\_11\_50\_10\_1 & 399 & 34 & 26 & 34 & \textbf{13} & 34 & 125 & 34 & 101 & 34 \\
R\_11\_50\_10\_2 & 488 & 52 & 27 & 52 & \textbf{15} & 52 & 162 & 52 & 121 & 52 \\
R\_11\_50\_20\_1 & 480 & 34 & 34 & 34 & \textbf{10} & 34 & 127 & 34 & 76 & 34 \\
R\_11\_50\_20\_2 & 1133 & 60 & 44 & 60 & \textbf{26} & 60 & 289 & 60 & 155 & 60 \\
R\_11\_70\_10\_1 & 918 & 51 & 32 & 51 & \textbf{15} & 51 & 149 & 51 & 107 & 51 \\
R\_11\_70\_10\_2 & 331 & 32 & 28 & 32 & \textbf{10} & 32 & 110 & 32 & 70 & 32 \\
R\_11\_70\_20\_1 & 706 & 35 & 60 & 35 & \textbf{10} & 35 & 133 & 35 & 84 & 35 \\
R\_11\_70\_20\_2 & 267 & 23 & 25 & 23 & \textbf{7} & 23 & 75 & 23 & 58 & 23 \\
R\_11\_90\_10\_1 & 208 & 26 & 14 & 26 & \textbf{7} & 26 & 55 & 26 & 53 & 26 \\
R\_11\_90\_10\_2 & 165 & 30 & 15 & 30 & \textbf{8} & 30 & 41 & 30 & 37 & 30 \\
R\_11\_90\_20\_1 & 535 & 65 & 42 & 65 & \textbf{10} & 65 & 151 & 65 & 120 & 65 \\
R\_11\_90\_20\_2 & 804 & 47 & 43 & 47 & \textbf{10} & 47 & 188 & 47 & 106 & 47 \\
R\_12\_50\_10\_1 & 4518 & 73 & 93 & 73 & \textbf{34} & 73 & 692 & 73 & 329 & 73 \\
R\_12\_50\_10\_2 & 1892 & 48 & 58 & 48 & \textbf{31} & 48 & 425 & 48 & 298 & 48 \\
R\_12\_50\_20\_1 & 4219 & 71 & 112 & 71 & \textbf{59} & 71 & 863 & 71 & 347 & 71 \\
R\_12\_50\_20\_2 & 3488 & 75 & 96 & 75 & \textbf{43} & 75 & 350 & 75 & 207 & 75 \\
R\_12\_70\_10\_1 & 829 & 35 & 51 & 35 & \textbf{18} & 35 & 274 & 35 & 203 & 35 \\
R\_12\_70\_10\_2 & 4792 & 86 & 98 & 86 & \textbf{26} & 86 & 988 & 86 & 485 & 86 \\
R\_12\_70\_20\_1 & 2153 & 61 & 91 & 61 & \textbf{19} & 61 & 549 & 61 & 337 & 61 \\
R\_12\_70\_20\_2 & 191 & 11 & 14 & 11 & \textbf{8} & 11 & 38 & 11 & 29 & 11 \\
R\_12\_90\_10\_1 & 2119 & 64 & 98 & 64 & \textbf{28} & 64 & 701 & 64 & 399 & 64 \\
R\_12\_90\_10\_2 & 1361 & 47 & 58 & 47 & \textbf{14} & 47 & 301 & 47 & 215 & 47 \\
R\_12\_90\_20\_1 & 2455 & 62 & 97 & 62 & \textbf{25} & 62 & 805 & 62 & 421 & 62 \\
R\_12\_90\_20\_2 & 1743 & 44 & 71 & 44 & \textbf{22} & 44 & 403 & 44 & 280 & 44 \\
R\_13\_50\_10\_1 & 2504 & 55 & 77 & 55 & \textbf{32} & 55 & 491 & 55 & 349 & 55 \\
R\_13\_50\_10\_2 & 11194 & 95 & 195 & 95 & \textbf{99} & 95 & 2549 & 95 & 1052 & 95 \\
R\_13\_50\_20\_1 & 4247 & 63 & 170 & 63 & \textbf{70} & 63 & 1922 & 63 & 826 & 63 \\
R\_13\_50\_20\_2 & 6020 & 66 & 195 & 66 & \textbf{55} & 66 & 2306 & 66 & 750 & 66 \\
R\_13\_70\_10\_1 & 2033 & 56 & 102 & 56 & \textbf{39} & 56 & 1108 & 56 & 473 & 56 \\
R\_13\_70\_10\_2 & 4050 & 64 & 132 & 64 & \textbf{32} & 64 & 1139 & 64 & 576 & 64 \\
R\_13\_70\_20\_1 & 1480 & 32 & 85 & 32 & \textbf{27} & 32 & 399 & 32 & 225 & 32 \\
R\_13\_70\_20\_2 & 1491 & 46 & 79 & 46 & \textbf{13} & 46 & 194 & 46 & 175 & 46 \\
R\_13\_90\_10\_1 & 772 & 26 & 53 & 26 & \textbf{28} & 26 & 279 & 26 & 150 & 26 \\
R\_13\_90\_10\_2 & 1913 & 65 & 111 & 65 & \textbf{21} & 65 & 471 & 65 & 405 & 65 \\
R\_13\_90\_20\_1 & 4629 & 94 & 233 & 94 & \textbf{53} & 94 & 2684 & 94 & 1167 & 94 \\
R\_13\_90\_20\_2 & 9399 & 94 & 414 & 94 & \textbf{122} & 94 & 3143 & 94 & 1376 & 94 \\
R\_14\_50\_10\_1 & 14587 & 82 & 500 & 82 & \textbf{232} & 82 & 6642 & 82 & 2700 & 82 \\
R\_14\_50\_10\_2 & 10952 & 81 & 272 & 81 & \textbf{102} & 81 & 4376 & 81 & 1543 & 81 \\
R\_14\_50\_20\_1 & 6974 & 74 & 259 & 74 & \textbf{69} & 74 & 3876 & 74 & 1384 & 74 \\
R\_14\_50\_20\_2 & TL & 84 & 433 & 119 & \textbf{112} & 119 & 8716 & 119 & 3165 & 119 \\
R\_14\_70\_10\_1 & TL & 74 & 406 & 117 & \textbf{109} & 117 & 6324 & 117 & 2568 & 117 \\
R\_14\_70\_10\_2 & TL & 102 & 473 & 108 & \textbf{188} & 108 & 5189 & 108 & 2624 & 108 \\
R\_14\_70\_20\_1 & TL & 87 & 521 & 111 & \textbf{194} & 111 & 9819 & 111 & 3438 & 111 \\
R\_14\_70\_20\_2 & TL & 71 & 427 & 103 & \textbf{61} & 103 & 4230 & 103 & 2008 & 103 \\
R\_14\_90\_10\_1 & 3950 & 47 & 212 & 47 & \textbf{56} & 47 & 2080 & 47 & 918 & 47 \\
R\_14\_90\_10\_2 & 3343 & 47 & 159 & 47 & \textbf{20} & 47 & 730 & 47 & 489 & 47 \\
R\_14\_90\_20\_1 & 4549 & 63 & 268 & 63 & \textbf{71} & 63 & 2431 & 63 & 1437 & 63 \\
R\_14\_90\_20\_2 & 4104 & 32 & 125 & 32 & \textbf{36} & 32 & 646 & 32 & 455 & 32 \\
R\_15\_50\_10\_1 & TL & 35 & 1093 & 157 & \textbf{550} & 157 & TL & 113 & 7632 & 157 \\
R\_15\_50\_10\_2 & TL & 60 & 442 & 91 & \textbf{136} & 91 & 7249 & 91 & 3440 & 91 \\
R\_15\_50\_20\_1 & TL & 41 & 1488 & 131 & \textbf{379} & 131 & 8997 & 131 & 3596 & 131 \\
R\_15\_50\_20\_2 & TL & 58 & 1190 & 107 & \textbf{378} & 107 & TL & 102 & 6519 & 107 \\
R\_15\_70\_10\_1 & TL & 36 & 675 & 77 & \textbf{246} & 77 & 10563 & 77 & 4227 & 77 \\
R\_15\_70\_10\_2 & TL & 42 & 523 & 73 & \textbf{111} & 73 & 8062 & 73 & 3623 & 73 \\
R\_15\_70\_20\_1 & 15577 & 84 & 431 & 84 & \textbf{131} & 84 & 6768 & 84 & 2850 & 84 \\
R\_15\_70\_20\_2 & TL & 77 & 738 & 98 & \textbf{201} & 98 & 5661 & 98 & 2062 & 98 \\
R\_15\_90\_10\_1 & TL & 59 & 696 & 112 & \textbf{160} & 112 & 12367 & 112 & 5908 & 112 \\
R\_15\_90\_10\_2 & 10583 & 66 & 299 & 66 & \textbf{63} & 66 & 3351 & 66 & 1697 & 66 \\
R\_15\_90\_20\_1 & TL & 35 & 545 & 53 & \textbf{109} & 53 & 11831 & 53 & 3655 & 53 \\
R\_15\_90\_20\_2 & TL & 42 & 392 & 42 & \textbf{124} & 42 & 4708 & 42 & 1617 & 42 \\
R\_16\_50\_10\_1 & 16437 & 92 & 997 & 92 & \textbf{162} & 92 & 12029 & 92 & 4977 & 92 \\
R\_16\_50\_10\_2 & TL & 13 & 1863 & 123 & \textbf{514} & 123 & TL & 91 & 10292 & 123 \\
R\_16\_50\_20\_1 & TL & 45 & 5195 & 219 & \textbf{2052} & 219 & TL & 73 & TL & 162 \\
R\_16\_50\_20\_2 & TL & 17 & 4149 & 183 & \textbf{2004} & 183 & TL & 42 & TL & 115 \\
R\_16\_70\_10\_1 & TL & 37 & 2004 & 116 & \textbf{332} & 116 & TL & 56 & 12363 & 116 \\
R\_16\_70\_10\_2 & TL & 35 & 1336 & 87 & \textbf{410} & 87 & TL & 66 & 10195 & 87 \\
R\_16\_70\_20\_1 & TL & 28 & 2113 & 155 & \textbf{869} & 155 & TL & 77 & 14478 & 155 \\
R\_16\_70\_20\_2 & TL & 54 & 954 & 91 & \textbf{415} & 91 & TL & 70 & 8277 & 91 \\
R\_16\_90\_10\_1 & TL & 34 & 568 & 79 & \textbf{115} & 79 & TL & 59 & 9755 & 79 \\
R\_16\_90\_10\_2 & TL & 52 & 884 & 107 & \textbf{142} & 107 & TL & 65 & 12424 & 107 \\
R\_16\_90\_20\_1 & TL & 52 & 1012 & 109 & \textbf{130} & 109 & TL & 68 & 11708 & 109 \\
R\_16\_90\_20\_2 & TL & 25 & 1541 & 126 & \textbf{295} & 126 & TL & 52 & TL & 125 \\
R\_17\_50\_10\_1 & TL & 18 & 4465 & 190 & \textbf{2477} & 190 & TL & 33 & TL & 109 \\
R\_17\_50\_10\_2 & TL & 9 & 2244 & 134 & \textbf{1627} & 134 & TL & 36 & 17318 & 134 \\
R\_17\_50\_20\_1 & TL & 2 & 4103 & 162 & \textbf{2588} & 162 & TL & 69 & 12518 & 162 \\
R\_17\_50\_20\_2 & TL & 28 & 3479 & 162 & \textbf{547} & 162 & TL & 60 & TL & 126 \\
R\_17\_70\_10\_1 & TL & 10 & 4790 & 151 & \textbf{1368} & 151 & TL & 37 & TL & 121 \\
R\_17\_70\_10\_2 & TL & 11 & 3033 & 218 & \textbf{919} & 218 & TL & 54 & TL & 113 \\
R\_17\_70\_20\_1 & TL & 5 & 4540 & 131 & \textbf{929} & 131 & TL & 51 & 12397 & 131 \\
R\_17\_70\_20\_2 & TL & 17 & 10546 & 207 & \textbf{3803} & 207 & TL & 31 & TL & 72 \\
R\_17\_90\_10\_1 & TL & 9 & 1441 & 105 & \textbf{438} & 105 & TL & 21 & TL & 42 \\
R\_17\_90\_10\_2 & TL & 18 & 2938 & 146 & \textbf{541} & 146 & TL & 38 & 10918 & 146 \\
R\_17\_90\_20\_1 & TL & 13 & 3618 & 122 & \textbf{684} & 122 & TL & 42 & 12045 & 122 \\
R\_17\_90\_20\_2 & TL & 18 & 1591 & 85 & \textbf{286} & 85 & TL & 25 & TL & 67 \\
R\_18\_50\_10\_1 & TL & 18 & 6770 & 175 & \textbf{2600} & 175 & TL & 17 & TL & 82 \\
R\_18\_50\_10\_2 & TL & 13 & 16156 & 145 & \textbf{2867} & 145 & TL & 17 & TL & 42 \\
R\_18\_50\_20\_1 & TL & 16 & 4417 & 135 & \textbf{1312} & 135 & TL & 28 & TL & 123 \\
R\_18\_50\_20\_2 & TL & 10 & TL & 242 & \textbf{12022} & 247 & TL & 36 & TL & 61 \\
R\_18\_70\_10\_1 & TL & 5 & 8332 & 179 & \textbf{1969} & 179 & TL & 34 & TL & 71 \\
R\_18\_70\_10\_2 & TL & 7 & 4041 & 176 & \textbf{1251} & 176 & TL & 38 & TL & 120 \\
R\_18\_70\_20\_1 & TL & 5 & 13635 & 245 & \textbf{4634} & 245 & TL & 23 & TL & 52 \\
R\_18\_70\_20\_2 & TL & 6 & 4306 & 119 & \textbf{1030} & 119 & TL & 31 & 12174 & 119 \\
R\_18\_90\_10\_1 & TL & 12 & 6348 & 223 & \textbf{1024} & 223 & TL & 25 & TL & 46 \\
R\_18\_90\_10\_2 & TL & 15 & 5578 & 154 & \textbf{1108} & 154 & TL & 24 & TL & 74 \\
R\_18\_90\_20\_1 & TL & 11 & 9006 & 144 & \textbf{1305} & 144 & TL & 26 & TL & 71 \\
R\_18\_90\_20\_2 & TL & 5 & 10455 & 255 & \textbf{2527} & 255 & TL & 14 & TL & 35 \\
R\_19\_50\_10\_1 & TL & 7 & TL & 196 & \textbf{7208} & 296 & TL & 21 & TL & 60 \\
R\_19\_50\_10\_2 & TL & 4 & TL & 167 & TL & \textbf{255} & TL & 20 & TL & 68 \\
R\_19\_50\_20\_1 & TL & 0 & TL & 313 & \textbf{4527} & 423 & TL & 12 & TL & 31 \\
R\_19\_50\_20\_2 & TL & 18 & 14027 & 268 & \textbf{1725} & 268 & TL & 31 & TL & 143 \\
R\_19\_70\_10\_1 & TL & 0 & 10170 & 246 & \textbf{2547} & 246 & TL & 16 & TL & 121 \\
R\_19\_70\_10\_2 & TL & 4 & TL & 248 & \textbf{3865} & 282 & TL & 24 & TL & 68 \\
R\_19\_70\_20\_1 & TL & 1 & TL & 113 & TL & \textbf{278} & TL & 15 & TL & 55 \\
R\_19\_70\_20\_2 & TL & 4 & TL & 284 & \textbf{4672} & 354 & TL & 27 & TL & 71 \\
R\_19\_90\_10\_1 & TL & 1 & 4927 & 125 & \textbf{1304} & 125 & TL & 5 & TL & 8 \\
R\_19\_90\_10\_2 & TL & 2 & 14494 & 260 & \textbf{4070} & 260 & TL & 10 & TL & 23 \\
R\_19\_90\_20\_1 & TL & 18 & 13743 & 206 & \textbf{2430} & 206 & TL & 10 & TL & 40 \\
R\_19\_90\_20\_2 & TL & 17 & 7349 & 142 & \textbf{1649} & 142 & TL & 13 & TL & 29 \\
R\_20\_50\_10\_1 & TL & 1 & TL & 107 & \textbf{11216} & 291 & TL & 12 & TL & 38 \\
R\_20\_50\_10\_2 & TL & 4 & TL & 45 & TL & \textbf{226} & TL & 10 & TL & 36 \\
R\_20\_50\_20\_1 & TL & 0 & TL & 54 & TL & \textbf{388} & TL & 17 & TL & 39 \\
R\_20\_50\_20\_2 & TL & 4 & TL & 109 & \textbf{5818} & 284 & TL & 11 & TL & 36 \\
R\_20\_70\_10\_1 & TL & 0 & TL & 134 & \textbf{13484} & 255 & TL & 29 & TL & 93 \\
R\_20\_70\_10\_2 & TL & 0 & TL & 103 & \textbf{7157} & 233 & TL & 7 & TL & 17 \\
R\_20\_70\_20\_1 & TL & 1 & TL & 113 & \textbf{7040} & 317 & TL & 4 & TL & 16 \\
R\_20\_70\_20\_2 & TL & 0 & TL & 124 & TL & \textbf{277} & TL & 7 & TL & 16 \\
R\_20\_90\_10\_1 & TL & 8 & TL & 131 & \textbf{6273} & 238 & TL & 8 & TL & 28 \\
R\_20\_90\_10\_2 & TL & 6 & 10263 & 151 & \textbf{1156} & 151 & TL & 6 & TL & 13 \\
R\_20\_90\_20\_1 & TL & 2 & 15465 & 210 & \textbf{2544} & 210 & TL & 3 & TL & 6 \\
R\_20\_90\_20\_2 & TL & 0 & TL & 176 & \textbf{4202} & 332 & TL & 9 & TL & 36 

%% file: tables/features/features_lit_table.tex
\begingroup\fontsize{9}{10}\selectfont
\begin{longtable}{lrrrrrrrrrr}
\caption{\label{tab:features_lit} Runtimes and numbers of nondominated points found for different settings of our algorithm for \texttt{combined} instances}\\
\toprule
& \multicolumn{2}{c}{\texttt{I}} & \multicolumn{2}{c}{\texttt{IS}} & \multicolumn{2}{c}{\texttt{ISR}} & \multicolumn{2}{c}{\texttt{ISRT}} &  \multicolumn{2}{c}{\texttt{ISRTP}} \\
\cmidrule(lr){2-3}\cmidrule(lr){4-5}\cmidrule(lr){6-7}\cmidrule(lr){8-9}\cmidrule(lr){10-11}
\multicolumn{1}{c}{instance} & \multicolumn{1}{c}{$t \,[s]$} & \multicolumn{1}{c}{$|\mathcal{Y}_N|$} & \multicolumn{1}{c}{$t \,[s]$} & \multicolumn{1}{c}{$|\mathcal{Y}_N|$} & \multicolumn{1}{c}{$t \,[s]$} & \multicolumn{1}{c}{$|\mathcal{Y}_N|$} & \multicolumn{1}{c}{$t \,[s]$} & \multicolumn{1}{c}{$|\mathcal{Y}_N|$}& \multicolumn{1}{c}{$t \,[s]$} & \multicolumn{1}{c}{$|\mathcal{Y}_N|$}\\ 
\midrule
\endfirsthead
\caption[]{Runtimes and numbers of nondominated points found for different settings of our algorithm for \texttt{combined} instances (continued)}\\
\toprule
& \multicolumn{2}{c}{\texttt{I}} & \multicolumn{2}{c}{\texttt{IS}} & \multicolumn{2}{c}{\texttt{ISR}} & \multicolumn{2}{c}{\texttt{ISRT}} &  \multicolumn{2}{c}{\texttt{ISRTP}} \\
\cmidrule(lr){2-3}\cmidrule(lr){4-5}\cmidrule(lr){6-7}\cmidrule(lr){8-9}\cmidrule(lr){10-11}
\multicolumn{1}{c}{instance} & \multicolumn{1}{c}{$t \,[s]$} & \multicolumn{1}{c}{$|\mathcal{Y}_N|$} & \multicolumn{1}{c}{$t \,[s]$} & \multicolumn{1}{c}{$|\mathcal{Y}_N|$} & \multicolumn{1}{c}{$t \,[s]$} & \multicolumn{1}{c}{$|\mathcal{Y}_N|$} & \multicolumn{1}{c}{$t \,[s]$} & \multicolumn{1}{c}{$|\mathcal{Y}_N|$}& \multicolumn{1}{c}{$t \,[s]$} & \multicolumn{1}{c}{$|\mathcal{Y}_N|$} \\ 
\midrule
\endhead
\midrule
\multicolumn{11}{r}{Continued on next page} \\
\midrule
\endfoot
\bottomrule
\endlastfoot
\input{tables/data/features_lit}%
\end{longtable}
\endgroup{}

%% file: tables/data/features_lit.tex
S9H\_SRFLP9 & 115 & 35 & 109 & 35 & 33 & 35 & 12 & 35 & \textbf{11} & 35 \\
S9\_S9H & 35 & 20 & 34 & 20 & 10 & 20 & \textbf{4} & 20 & \textbf{4} & 20 \\
S9H\_S9 & 27 & 16 & 24 & 16 & 8 & 16 & \textbf{4} & 16 & \textbf{4} & 16 \\
S9\_SRFLP9 & 64 & 21 & 63 & 21 & 17 & 21 & \textbf{8} & 21 & \textbf{8} & 21 \\
SRFLP9\_S9 & 37 & 10 & 37 & 10 & 10 & 10 & \textbf{5} & 10 & \textbf{5} & 10 \\
SRFLP9\_S9H & 49 & 18 & 48 & 18 & 14 & 18 & \textbf{5} & 18 & \textbf{5} & 18 \\
SRFLP10\_S10 & 106 & 21 & 103 & 21 & 28 & 21 & \textbf{9} & 21 & \textbf{9} & 21 \\
S10\_SRFLP10 & 147 & 18 & 134 & 18 & 28 & 18 & \textbf{11} & 18 & 12 & 18 \\
S11\_LW11 & 650 & 46 & 608 & 46 & 172 & 46 & 44 & 46 & \textbf{43} & 46 \\
LW11\_S11 & 1419 & 79 & 1324 & 79 & 384 & 79 & \textbf{85} & 79 & \textbf{85} & 79 \\
H15\_SRFLP15 & 12481 & 62 & 10381 & 62 & 3025 & 62 & 890 & 62 & \textbf{884} & 62 \\
AM15\_SRFLP15 & 17067 & 89 & 14248 & 89 & 4489 & 89 & 1612 & 89 & \textbf{1564} & 89 \\
SRFLP15\_AM15 & TL & 82 & TL & 85 & 6424 & 110 & 1536 & 110 & \textbf{1486} & 110 

%% file: tables/features/features_rand_table.tex
\begingroup\fontsize{9}{10}\selectfont
\begin{longtable}{lrrrrrrrrrr}
\caption{\label{tab:features_rand} Runtimes and numbers of nondominated points found for different settings of our algorithm for \texttt{random} instances}\\
\toprule
& \multicolumn{2}{c}{\texttt{I}} & \multicolumn{2}{c}{\texttt{IS}} & \multicolumn{2}{c}{\texttt{ISR}} & \multicolumn{2}{c}{\texttt{ISRT}} &  \multicolumn{2}{c}{\texttt{ISRTP}} \\
\cmidrule(lr){2-3}\cmidrule(lr){4-5}\cmidrule(lr){6-7}\cmidrule(lr){8-9}\cmidrule(lr){10-11}
\multicolumn{1}{c}{instance} & \multicolumn{1}{c}{$t \,[s]$} & \multicolumn{1}{c}{$|\mathcal{Y}_N|$} & \multicolumn{1}{c}{$t \,[s]$} & \multicolumn{1}{c}{$|\mathcal{Y}_N|$} & \multicolumn{1}{c}{$t \,[s]$} & \multicolumn{1}{c}{$|\mathcal{Y}_N|$} & \multicolumn{1}{c}{$t \,[s]$} & \multicolumn{1}{c}{$|\mathcal{Y}_N|$}& \multicolumn{1}{c}{$t \,[s]$} & \multicolumn{1}{c}{$|\mathcal{Y}_N|$}\\ 
\midrule
\endfirsthead
\caption{Runtimes and numbers of nondominated points found for different settings of our algorithm for \texttt{random} instances (continued)}\\
\toprule
 & \multicolumn{2}{c}{\texttt{I}} & \multicolumn{2}{c}{\texttt{IS}} & \multicolumn{2}{c}{\texttt{ISR}} & \multicolumn{2}{c}{\texttt{ISRT}} &  \multicolumn{2}{c}{\texttt{ISRTP}} \\
 \cmidrule(lr){2-3}\cmidrule(lr){4-5}\cmidrule(lr){6-7}\cmidrule(lr){8-9}\cmidrule(lr){10-11}
\multicolumn{1}{c}{instance} & \multicolumn{1}{c}{$t \,[s]$} & \multicolumn{1}{c}{$|\mathcal{Y}_N|$} & \multicolumn{1}{c}{$t \,[s]$} & \multicolumn{1}{c}{$|\mathcal{Y}_N|$} & \multicolumn{1}{c}{$t \,[s]$} & \multicolumn{1}{c}{$|\mathcal{Y}_N|$} & \multicolumn{1}{c}{$t \,[s]$} & \multicolumn{1}{c}{$|\mathcal{Y}_N|$}& \multicolumn{1}{c}{$t \,[s]$} & \multicolumn{1}{c}{$|\mathcal{Y}_N|$} \\ 
\midrule
\endhead
\midrule
\multicolumn{11}{r}{Continued on next page} \\
\midrule
\endfoot
\bottomrule
\endlastfoot
\input{tables/data/features_rand}%
\end{longtable}
\endgroup{}

%% file: tables/data/features_rand.tex
R\_10\_50\_10\_1 & 98 & 23 & 96 & 23 & 27 & 23 & 10 & 23 & \textbf{9} & 23 \\
R\_10\_50\_10\_2 & 337 & 41 & 320 & 41 & 81 & 41 & \textbf{21} & 41 & \textbf{21} & 41 \\
R\_10\_50\_20\_1 & 136 & 22 & 130 & 22 & 34 & 22 & \textbf{9} & 22 & \textbf{9} & 22 \\
R\_10\_50\_20\_2 & 187 & 27 & 173 & 27 & 46 & 27 & 18 & 27 & \textbf{17} & 27 \\
R\_10\_70\_10\_1 & 67 & 18 & 64 & 18 & 19 & 18 & \textbf{9} & 18 & \textbf{9} & 18 \\
R\_10\_70\_10\_2 & 73 & 17 & 68 & 17 & 20 & 17 & \textbf{8} & 17 & \textbf{8} & 17 \\
R\_10\_70\_20\_1 & 278 & 26 & 249 & 26 & 62 & 26 & 24 & 26 & \textbf{23} & 26 \\
R\_10\_70\_20\_2 & 316 & 33 & 299 & 33 & 78 & 33 & 23 & 33 & \textbf{22} & 33 \\
R\_10\_90\_10\_1 & 132 & 24 & 132 & 24 & 38 & 24 & 14 & 24 & \textbf{13} & 24 \\
R\_10\_90\_10\_2 & 345 & 38 & 318 & 38 & 87 & 38 & \textbf{28} & 38 & \textbf{28} & 38 \\
R\_10\_90\_20\_1 & 106 & 21 & 103 & 21 & 31 & 21 & 10 & 21 & \textbf{9} & 21 \\
R\_10\_90\_20\_2 & 153 & 13 & 146 & 13 & 39 & 13 & \textbf{16} & 13 & 18 & 13 \\
R\_11\_50\_10\_1 & 399 & 34 & 382 & 34 & 94 & 34 & \textbf{31} & 34 & \textbf{31} & 34 \\
R\_11\_50\_10\_2 & 488 & 52 & 471 & 52 & 136 & 52 & 33 & 52 & \textbf{32} & 52 \\
R\_11\_50\_20\_1 & 480 & 34 & 463 & 34 & 118 & 34 & \textbf{43} & 34 & \textbf{43} & 34 \\
R\_11\_50\_20\_2 & 1133 & 60 & 1044 & 60 & 252 & 60 & \textbf{59} & 60 & 61 & 60 \\
R\_11\_70\_10\_1 & 918 & 51 & 810 & 51 & 218 & 51 & \textbf{43} & 51 & \textbf{43} & 51 \\
R\_11\_70\_10\_2 & 331 & 32 & 313 & 32 & 91 & 32 & 33 & 32 & \textbf{32} & 32 \\
R\_11\_70\_20\_1 & 706 & 35 & 655 & 35 & 176 & 35 & \textbf{58} & 35 & 61 & 35 \\
R\_11\_70\_20\_2 & 267 & 23 & 255 & 23 & 76 & 23 & \textbf{34} & 23 & 35 & 23 \\
R\_11\_90\_10\_1 & 208 & 26 & 209 & 26 & 64 & 26 & 25 & 26 & \textbf{24} & 26 \\
R\_11\_90\_10\_2 & 165 & 30 & 164 & 30 & 57 & 30 & \textbf{17} & 30 & \textbf{17} & 30 \\
R\_11\_90\_20\_1 & 535 & 65 & 526 & 65 & 179 & 65 & \textbf{48} & 65 & \textbf{48} & 65 \\
R\_11\_90\_20\_2 & 804 & 47 & 750 & 47 & 214 & 47 & 61 & 47 & \textbf{54} & 47 \\
R\_12\_50\_10\_1 & 4518 & 73 & 4230 & 73 & 804 & 73 & \textbf{109} & 73 & \textbf{109} & 73 \\
R\_12\_50\_10\_2 & 1892 & 48 & 1713 & 48 & 404 & 48 & \textbf{74} & 48 & 75 & 48 \\
R\_12\_50\_20\_1 & 4219 & 71 & 3698 & 71 & 908 & 71 & \textbf{226} & 71 & 233 & 71 \\
R\_12\_50\_20\_2 & 3488 & 75 & 2863 & 75 & 715 & 75 & \textbf{106} & 75 & 112 & 75 \\
R\_12\_70\_10\_1 & 829 & 35 & 765 & 35 & 209 & 35 & \textbf{71} & 35 & 72 & 35 \\
R\_12\_70\_10\_2 & 4792 & 86 & 4344 & 86 & 998 & 86 & \textbf{179} & 86 & 183 & 86 \\
R\_12\_70\_20\_1 & 2153 & 61 & 1976 & 61 & 523 & 61 & 135 & 61 & \textbf{129} & 61 \\
R\_12\_70\_20\_2 & 191 & 11 & 171 & 11 & 53 & 11 & \textbf{21} & 11 & \textbf{21} & 11 \\
R\_12\_90\_10\_1 & 2119 & 64 & 1887 & 64 & 504 & 64 & 133 & 64 & \textbf{124} & 64 \\
R\_12\_90\_10\_2 & 1361 & 47 & 1253 & 47 & 346 & 47 & 83 & 47 & \textbf{78} & 47 \\
R\_12\_90\_20\_1 & 2455 & 62 & 2228 & 62 & 612 & 62 & \textbf{153} & 62 & 154 & 62 \\
R\_12\_90\_20\_2 & 1743 & 44 & 1537 & 44 & 397 & 44 & 98 & 44 & \textbf{94} & 44 \\
R\_13\_50\_10\_1 & 2504 & 55 & 2300 & 55 & 622 & 55 & \textbf{124} & 55 & \textbf{124} & 55 \\
R\_13\_50\_10\_2 & 11194 & 95 & 10016 & 95 & 2198 & 95 & \textbf{348} & 95 & 350 & 95 \\
R\_13\_50\_20\_1 & 4247 & 63 & 3728 & 63 & 1000 & 63 & \textbf{211} & 63 & 214 & 63 \\
R\_13\_50\_20\_2 & 6020 & 66 & 5446 & 66 & 1396 & 66 & \textbf{352} & 66 & 366 & 66 \\
R\_13\_70\_10\_1 & 2033 & 56 & 1809 & 56 & 501 & 56 & 155 & 56 & \textbf{152} & 56 \\
R\_13\_70\_10\_2 & 4050 & 64 & 3467 & 64 & 946 & 64 & 218 & 64 & \textbf{207} & 64 \\
R\_13\_70\_20\_1 & 1480 & 32 & 1405 & 32 & 403 & 32 & \textbf{113} & 32 & 117 & 32 \\
R\_13\_70\_20\_2 & 1491 & 46 & 1336 & 46 & 409 & 46 & \textbf{111} & 46 & 112 & 46 \\
R\_13\_90\_10\_1 & 772 & 26 & 719 & 26 & 232 & 26 & 68 & 26 & \textbf{62} & 26 \\
R\_13\_90\_10\_2 & 1913 & 65 & 1765 & 65 & 537 & 65 & 151 & 65 & \textbf{144} & 65 \\
R\_13\_90\_20\_1 & 4629 & 94 & 4237 & 94 & 1215 & 94 & 291 & 94 & \textbf{274} & 94 \\
R\_13\_90\_20\_2 & 9399 & 94 & 8032 & 94 & 2056 & 94 & 440 & 94 & \textbf{437} & 94 \\
R\_14\_50\_10\_1 & 14587 & 82 & 13038 & 82 & 3292 & 82 & 523 & 82 & \textbf{498} & 82 \\
R\_14\_50\_10\_2 & 10952 & 81 & 10490 & 81 & 2799 & 81 & 445 & 81 & \textbf{425} & 81 \\
R\_14\_50\_20\_1 & 6974 & 74 & 6240 & 74 & 1732 & 74 & 304 & 74 & \textbf{300} & 74 \\
R\_14\_50\_20\_2 & TL & 84 & TL & 109 & 4995 & 119 & 947 & 119 & \textbf{854} & 119 \\
R\_14\_70\_10\_1 & TL & 74 & TL & 79 & 6000 & 117 & \textbf{699} & 117 & 727 & 117 \\
R\_14\_70\_10\_2 & TL & 102 & 15782 & 108 & 4231 & 108 & 934 & 108 & \textbf{908} & 108 \\
R\_14\_70\_20\_1 & TL & 87 & TL & 101 & 5481 & 111 & \textbf{981} & 111 & 997 & 111 \\
R\_14\_70\_20\_2 & TL & 71 & TL & 90 & 5302 & 103 & 991 & 103 & \textbf{941} & 103 \\
R\_14\_90\_10\_1 & 3950 & 47 & 3742 & 47 & 1182 & 47 & \textbf{288} & 47 & 293 & 47 \\
R\_14\_90\_10\_2 & 3343 & 47 & 3034 & 47 & 955 & 47 & 211 & 47 & \textbf{210} & 47 \\
R\_14\_90\_20\_1 & 4549 & 63 & 4005 & 63 & 1206 & 63 & 283 & 63 & \textbf{279} & 63 \\
R\_14\_90\_20\_2 & 4104 & 32 & 3500 & 32 & 1002 & 32 & \textbf{189} & 32 & 195 & 32 \\
R\_15\_50\_10\_1 & TL & 35 & TL & 45 & 15679 & 157 & 2980 & 157 & \textbf{2976} & 157 \\
R\_15\_50\_10\_2 & TL & 60 & 16738 & 91 & 4591 & 91 & 1343 & 91 & \textbf{1273} & 91 \\
R\_15\_50\_20\_1 & TL & 41 & TL & 52 & 13302 & 131 & \textbf{2431} & 131 & 2493 & 131 \\
R\_15\_50\_20\_2 & TL & 58 & TL & 70 & 7079 & 107 & \textbf{1769} & 107 & 1847 & 107 \\
R\_15\_70\_10\_1 & TL & 36 & TL & 54 & 7066 & 77 & 1733 & 77 & \textbf{1693} & 77 \\
R\_15\_70\_10\_2 & TL & 42 & TL & 44 & 9519 & 73 & 612 & 73 & \textbf{611} & 73 \\
R\_15\_70\_20\_1 & 15577 & 84 & 13792 & 84 & 4043 & 84 & 829 & 84 & \textbf{826} & 84 \\
R\_15\_70\_20\_2 & TL & 77 & 17748 & 98 & 5696 & 98 & 1015 & 98 & \textbf{960} & 98 \\
R\_15\_90\_10\_1 & TL & 59 & TL & 64 & 7971 & 112 & 1387 & 112 & \textbf{1285} & 112 \\
R\_15\_90\_10\_2 & 10583 & 66 & 9058 & 66 & 2795 & 66 & \textbf{368} & 66 & 374 & 66 \\
R\_15\_90\_20\_1 & TL & 35 & TL & 38 & 6680 & 53 & \textbf{916} & 53 & 929 & 53 \\
R\_15\_90\_20\_2 & TL & 42 & 13316 & 42 & 3548 & 42 & \textbf{862} & 42 & 900 & 42 \\
R\_16\_50\_10\_1 & 16437 & 92 & 14591 & 92 & 4607 & 92 & 1326 & 92 & \textbf{1225} & 92 \\
R\_16\_50\_10\_2 & TL & 13 & TL & 15 & TL & 72 & 4791 & 123 & \textbf{4654} & 123 \\
R\_16\_50\_20\_1 & TL & 45 & TL & 51 & TL & 114 & \textbf{7982} & 219 & 8230 & 219 \\
R\_16\_50\_20\_2 & TL & 17 & TL & 19 & TL & 40 & \textbf{8098} & 183 & 8325 & 183 \\
R\_16\_70\_10\_1 & TL & 37 & TL & 39 & 15197 & 116 & 2796 & 116 & \textbf{2726} & 116 \\
R\_16\_70\_10\_2 & TL & 35 & TL & 38 & 12880 & 87 & \textbf{2353} & 87 & 2375 & 87 \\
R\_16\_70\_20\_1 & TL & 28 & TL & 34 & TL & 76 & \textbf{4684} & 155 & 4820 & 155 \\
R\_16\_70\_20\_2 & TL & 54 & TL & 62 & 8297 & 91 & \textbf{1701} & 91 & 1792 & 91 \\
R\_16\_90\_10\_1 & TL & 34 & TL & 40 & 8539 & 79 & 2135 & 79 & \textbf{2069} & 79 \\
R\_16\_90\_10\_2 & TL & 52 & TL & 53 & 16214 & 107 & \textbf{2298} & 107 & 2358 & 107 \\
R\_16\_90\_20\_1 & TL & 52 & TL & 56 & 11763 & 109 & \textbf{2053} & 109 & 2120 & 109 \\
R\_16\_90\_20\_2 & TL & 25 & TL & 28 & TL & 106 & \textbf{3026} & 126 & 3206 & 126 \\
R\_17\_50\_10\_1 & TL & 18 & TL & 20 & TL & 80 & 7717 & 190 & \textbf{7605} & 190 \\
R\_17\_50\_10\_2 & TL & 9 & TL & 11 & TL & 46 & 5868 & 134 & \textbf{5226} & 134 \\
R\_17\_50\_20\_1 & TL & 2 & TL & 3 & TL & 13 & 9981 & 162 & \textbf{9575} & 162 \\
R\_17\_50\_20\_2 & TL & 28 & TL & 35 & TL & 77 & 6533 & 162 & \textbf{6416} & 162 \\
R\_17\_70\_10\_1 & TL & 10 & TL & 14 & TL & 66 & \textbf{8651} & 151 & 8968 & 151 \\
R\_17\_70\_10\_2 & TL & 11 & TL & 17 & TL & 44 & 8917 & 218 & \textbf{8786} & 218 \\
R\_17\_70\_20\_1 & TL & 5 & TL & 8 & TL & 23 & \textbf{7811} & 131 & 7942 & 131 \\
R\_17\_70\_20\_2 & TL & 17 & TL & 23 & TL & 46 & 16407 & 207 & \textbf{16276} & 207 \\
R\_17\_90\_10\_1 & TL & 9 & TL & 14 & TL & 33 & 7618 & 105 & \textbf{7126} & 105 \\
R\_17\_90\_10\_2 & TL & 18 & TL & 24 & TL & 70 & 5147 & 146 & \textbf{4657} & 146 \\
R\_17\_90\_20\_1 & TL & 13 & TL & 16 & TL & 43 & 5491 & 122 & \textbf{4875} & 122 \\
R\_17\_90\_20\_2 & TL & 18 & TL & 25 & TL & 58 & \textbf{4778} & 85 & 4817 & 85 \\
R\_18\_50\_10\_1 & TL & 18 & TL & 21 & TL & 51 & 9713 & 175 & \textbf{8488} & 175 \\
R\_18\_50\_10\_2 & TL & 13 & TL & 14 & TL & 22 & TL & \textbf{47} & TL & 37 \\
R\_18\_50\_20\_1 & TL & 16 & TL & 20 & TL & 54 & \textbf{9714} & 135 & 10073 & 135 \\
R\_18\_50\_20\_2 & TL & 10 & TL & 13 & TL & 32 & TL & \textbf{80} & TL & 74 \\
R\_18\_70\_10\_1 & TL & 5 & TL & 8 & TL & 27 & TL & \textbf{103} & TL & 99 \\
R\_18\_70\_10\_2 & TL & 7 & TL & 12 & TL & 30 & 9501 & 176 & \textbf{9471} & 176 \\
R\_18\_70\_20\_1 & TL & 5 & TL & 11 & TL & 34 & TL & \textbf{229} & TL & 219 \\
R\_18\_70\_20\_2 & TL & 6 & TL & 10 & TL & 47 & 7410 & 119 & \textbf{7251} & 119 \\
R\_18\_90\_10\_1 & TL & 12 & TL & 14 & TL & 26 & TL & \textbf{125} & TL & 117 \\
R\_18\_90\_10\_2 & TL & 15 & TL & 19 & TL & 35 & \textbf{13448} & 154 & 14248 & 154 \\
R\_18\_90\_20\_1 & TL & 11 & TL & 13 & TL & 58 & TL & \textbf{121} & TL & 115 \\
R\_18\_90\_20\_2 & TL & 5 & TL & 7 & TL & 22 & TL & \textbf{126} & TL & 116 \\
R\_19\_50\_10\_1 & TL & 7 & TL & 13 & TL & 38 & TL & \textbf{127} & TL & 123 \\
R\_19\_50\_10\_2 & TL & 4 & TL & 6 & TL & 21 & TL & \textbf{135} & TL & 134 \\
R\_19\_50\_20\_1 & TL & 0 & TL & 1 & TL & 3 & TL & \textbf{42} & TL & 41 \\
R\_19\_50\_20\_2 & TL & 18 & TL & 21 & TL & 50 & TL & \textbf{153} & TL & 152 \\
R\_19\_70\_10\_1 & TL & 0 & TL & 1 & TL & 7 & TL & 74 & TL & \textbf{75} \\
R\_19\_70\_10\_2 & TL & 4 & TL & 6 & TL & 16 & TL & 75 & TL & \textbf{77} \\
R\_19\_70\_20\_1 & TL & 1 & TL & 2 & TL & 18 & TL & \textbf{85} & TL & \textbf{85} \\
R\_19\_70\_20\_2 & TL & 4 & TL & 5 & TL & 19 & TL & \textbf{126} & TL & 125 \\
R\_19\_90\_10\_1 & TL & 1 & TL & 2 & TL & 12 & 16145 & 125 & \textbf{16079} & 125 \\
R\_19\_90\_10\_2 & TL & 2 & TL & 4 & TL & 8 & TL & \textbf{29} & TL & \textbf{29} \\
R\_19\_90\_20\_1 & TL & 18 & TL & 19 & TL & 35 & TL & \textbf{106} & TL & 103 \\
R\_19\_90\_20\_2 & TL & 17 & TL & 19 & TL & 31 & TL & 119 & TL & \textbf{122} \\
R\_20\_50\_10\_1 & TL & 1 & TL & 1 & TL & 3 & TL & \textbf{22} & TL & 20 \\
R\_20\_50\_10\_2 & TL & 4 & TL & 5 & TL & 10 & TL & \textbf{39} & TL & \textbf{39} \\
R\_20\_50\_20\_1 & TL & 0 & TL & 1 & TL & 3 & TL & \textbf{32} & TL & 31 \\
R\_20\_50\_20\_2 & TL & 4 & TL & 6 & TL & 17 & TL & \textbf{42} & TL & 41 \\
R\_20\_70\_10\_1 & TL & 0 & TL & 1 & TL & 4 & TL & 59 & TL & \textbf{61} \\
R\_20\_70\_10\_2 & TL & 0 & TL & 0 & TL & 2 & TL & \textbf{6} & TL & \textbf{6} \\
R\_20\_70\_20\_1 & TL & 1 & TL & 2 & TL & 7 & TL & \textbf{51} & TL & 49 \\
R\_20\_70\_20\_2 & TL & 0 & TL & 0 & TL & 2 & TL & \textbf{10} & TL & \textbf{10} \\
R\_20\_90\_10\_1 & TL & 8 & TL & 9 & TL & 15 & TL & \textbf{28} & TL & 25 \\
R\_20\_90\_10\_2 & TL & 6 & TL & 7 & TL & 13 & TL & \textbf{34} & TL & 30 \\
R\_20\_90\_20\_1 & TL & 2 & TL & 2 & TL & 8 & TL & 117 & TL & \textbf{119} \\
R\_20\_90\_20\_2 & TL & 0 & TL & 1 & TL & 4 & TL & \textbf{35} & TL & 34 